\documentclass[12pt,a4paper,oneside]{amsart}

\usepackage{amsthm}
\usepackage{amsmath}
\usepackage{amssymb}
\usepackage{amscd}
\usepackage[utf8]{inputenc}
\usepackage{typearea}
\usepackage{eufrak}
\usepackage{yfonts}
\usepackage{textcomp}
\usepackage{mathrsfs}
\usepackage{hyperref}
\usepackage[draft]{fixme} % erase draft to disable fixme notes
\usepackage{pdfsync}
\usepackage[active]{srcltx}
\usepackage{xypic}
\usepackage{verbatim}

 \newcommand{\norm}[1]{\left\lVert#1\right\rVert}

 \theoremstyle{plain}
 \newtheorem{theorem}{Theorem}[section]
 \newtheorem{corollary}[theorem]{Corollary}
 \newtheorem{lemma}[theorem]{Lemma}
 \newtheorem{proposition}[theorem]{Proposition}
 \newtheorem{remark}[theorem]{Remark}
 \newtheorem{example}[theorem]{Example}
 
  \newtheorem{notation}[theorem]{Notation}
 
 \theoremstyle{definition}
 \newtheorem{definition}[theorem]{Definition}

\title[Twisted $L^p$-operator algebras]{Rigidity of twisted groupoid $L^p$-operator algebras.}

\begin{document}

\author{Einar V. Hetland}
\address{Department of Mathematical Sciences\\
	NTNU\\
	NO-7491 Trondheim\\
	Norway } \email{einarhetland96@gmail.com}

\author{Eduard Ortega}
\address{Department of Mathematical Sciences\\
NTNU\\
NO-7491 Trondheim\\
Norway } \email{eduard.ortega@ntnu.no}

%\dedicatory{Draft, version 1.1}



\date{\today}

%\maketitle
 
 \begin{abstract}
 	In this paper we   study the isomorphism problem for  reduced twisted group and groupoid $L^p$-operator algebras. For a locally compact group $G$ and a continuous 2-cocycle $\sigma$ we  define the reduced $\sigma$-twisted $L^p$-operator algebra $F_\lambda^p(G,\sigma)$. We  show that if $p\in (1,\infty)\setminus\{2\}$, then two such algebras are isometrically isomorphic if and only if the groups are topologically isomorphic and the continuous 2-cocyles are cohomologous. For a twist $\mathcal{E}$ over an étale groupoid $\mathcal{G}$, we define the reduced twisted groupoid $L^p$-operator algebra $F^p_\lambda(\mathcal{G};\mathcal{E})$. In the main result of this paper, we show that for $p\in[1,\infty)\setminus\{2\}$ if the groupoids are topologically principal, Hausdorff, étale and have a compact unit space, then two such algebras are isometrically isomorphic if and only if the groupoids are isomorphic and the twists are properly isomorphic.
 \end{abstract}
 
 \maketitle

 \section{Introduction}
 
 The study of operators on Hilbert spaces, in particular $C^\ast$-algebras, is one of the most active fields within 
 functional analysis today and there is enormous literature on $C^\ast$-algebras. The term $C^\ast$-algebra was first coined in 1947 by Segal to describe norm closed self-adjoint subalgebras of the bounded operators on some Hilbert space. A natural 
 generalization of this is to consider closed  non-selfadjoint algebras and replace the 
 Hilbert space, which is an $L^2$-space, with a general $L^p$-spaces. Given $p \in [1,\infty]$, we define an $L^p$-operator algebra to be a Banach
 algebra that that admits an isometric homomorphism to the space of bounded linear operators on some $L^p$-space. Hilbert spaces are very well behaved compared with general $L^p$-spaces and the much more complex geometry of $L^p$-spaces due to the lack of inner product when $p\neq 2$ makes $L^p$-operator algebras 
 usually much harder to study, even for $L^p$-operator algebras 
 that "look" like $C^\ast$-algebras. Many basic tools available in the $C^\ast$-algebra setting fail to hold for $L^p$-operator algebras. There exists for instance no general theory of $L^p$-operator algebras, and there is no abstract characterization of when a Banach algebra is an $L^p$-operator algebra. For these reasons the field has for the most part been studied through examples. However, more recent works have approached the field more systematically and abstractly, showing drips of a more general theory. 
 
 The field of $L^p$-operator algebras has seen a recent renewal in interest, spurred by new ideas and techniques taken from the field of operator algebras. There has been especially fruitful research on $L^p$-operator algebras that look like $C^\ast$-algebras, sparked by Christopher Phillips studies on the the $L^p$-analogues of the Cuntz algebra $\mathcal{O}_n$ and UHF-algebras \cite{Phi1,Phi2}.
 The more complex geometry of the unit ball of $L^p$-spaces, compared to $L^2$-spaces, allows  to prove some interesting rigidity results \cite{GarLup}. Besides the rigidity results, in \cite{Gardella_Thiel_16}  it is provided an unexpected solution to Le Merdy's problem for quotients of $L^p$-operator algebras. 
 
 %In this paper we will follow this line of research, following some of the recent work on $L^p$-analogue of the group and groupoid $C^\ast$-algebra \cite{choi2021rigidity, gardella2019modern, gardella2018isomorphisms}.
 
 For a locally compact group $G$, we denote the reduced group $L^p$-operator algebra by $F_\lambda^p(G)$ when $p=2$ this is the the reduced group $C^*$-algebra $C_\lambda^\ast(G)$. We will therefore say $F_\lambda^p(G)$ is the $L^p$-analogue of the reduced group  $C^\ast$-algebra. These algebras were first introduced by Hertz in the 1970s, under the name $p$-pseudofunctions. In the last decade there has been a renewed interest in these algebras and other group algebras on $L^p$-spaces, in particular of our interest is the result by Gardella and Thiel in \cite{gardella2018isomorphisms} which shows that, for $p\in (1,\infty)\setminus\{2\}$, two such algebras are isometrically isomorphic if and only if the groups are topological isomorphic. Surprisingly, this contrasts with the reduced group $C^\ast$-algebras where two non-isomorphic groups can generate isomorphic, and hence isometric, algebras. 
 
 Another interesting rigidity result comes from $L^p$-operator algebras associated to  étale  groupoids \cite{Lupin2014}. Given an étale groupoid one can define the reduced groupoid $L^p$-operator algebra $F^p_\lambda(\mathcal{G})$ so when $p=2$ it coincides with the reduced groupoid $C^*$-algebra defined in \cite{RenaultJean1980AGAt}. In \cite{choi2021rigidity}, the authors prove that the reduced groupoid $L^p$-operator algebras of two topological principal étale Hausdorff groupoids are isometrically isomorphic if and only if the groupoids are isomorphic. We would like to emphasize that even though groups are groupoids, the only group satisfying as a groupoid these hypothesis is the trivial group.

 In this paper, we will be interested in the $L^p$-operator algebras arising from the projective, or twisted, representations of both groups and groupoids. These kind of representations give rise to very interesting classes of operator algebras, where the noncommutative torus \cite{Rieffel} is among the most popular ones. 
 First, given a continuous 2-cocycle $\sigma$ on a locally compact group $G$ we will define the $\sigma$-twisted reduced $L^p$-operator algebra denoted by $F_\lambda^p(G,\sigma)$. For $p=2$, this is the $\sigma$-twisted reduced group $C^\ast$-algebra, which is a very well studied $C^*$-algebra \cite{Packer}. In the first main result  of this paper (Theorem \ref{main-group}) we will extend the isomorphism result of Gardella and Thiel to the reduced twisted group $L^p$-operator algebra and show that two such algebras (with $p\in (1,\infty)\setminus \{2\}$) are isometrically isomorphic if and only if the groups are topological isomorphic and the 2-cocycles are cohomologous. Again, this contrasts with the case when $p=2$, where two non-cohomologous cocyles on the same  group can give rise to isomorphic algebras \cite[Theorem 2]{Rieffel1}.

 Similarly to the group algebra setting, given a normalized continuous cocycle $\sigma$ on a Hausdorff étale groupoid $\mathcal{G}$, one can define the reduced twisted groupoid $L^p$-operator algebra denoted by $F^p_\lambda(\mathcal{G},\sigma)$, but formally more general, one can define the notion of twisted reduced groupoid $L^p$-operator algebra from a twist, based on Kumjian's work on the groupoid $C^\ast$-algebra \cite{kumjian_1986}. For a twist $\mathcal{E}$ over an étale groupoid $\mathcal{G}$ we define the twisted reduced groupoid $L^p$-operator algebra denoted by $F_\lambda^p(\mathcal{G};\mathcal{E})$. In \cite{renault2008cartan}, Renault shows that a topologically principal groupoid with a twist can be recovered from its twisted reduced groupoid $C^\ast$-algebra together with the canonical abelian subalgebra. Following Renault's work in \cite{renault2008cartan} and Choi, Gardella, and Thiel's \cite{choi2021rigidity}, we will prove our second main result (Theorem \ref{main_result}), showing that for a twist $\mathcal{E}$ over a topologically principal, Hausdorff, étale groupoid $\mathcal{G}$, one can recover the groupoid and the twist from the twisted reduced groupoid $L^p$-operator algebra when $p \neq 2$. This implies that if the groupoids are topologically principal Hausdorff, étale and $p\neq 2$, then two such algebras are isometrically isomorphic if and only if the groupoids are isomorphic and the twists are properly isomorphic.
 
 We would like to mention that even though some of the results and proofs are extensions of the ones in the non-twisted case, they are,  in most of the cases, far from being trivial and sometimes require more elaborate techniques to prove them (see fo example Theorem \ref{amenable_twisted} where we use a Fourier decomposition of the $L^p$-functions of the Mackey group of an amenable group to prove that the reduced norm of the twisted algebra of the group is the maximal possible norm).
 
 The paper is structured as follows. In Section $2$, we give a quick introduction to measures on Boolean algebras and isometries on $L^p$-spaces over measure Boolean algebras. The key point to all the rigidity results for $L^p$-operator algebras (for $p\neq 2$) is Lamperti's Theorem, that says that isometries on $L^p$-spaces are a composition of a scalar multiplication by a unitary $L^\infty$-function and a Boolean algebras transformation in the underlying measurable Boolean algebra. In Section $3$, given a locally compact group $G$ and a  continuous 2-cocycle $\sigma$ on $G$, we define the left regular $\sigma$-twisted representation of $G$ and define the $\sigma$-twisted group $L^p$-operator algebra $F_\lambda^p(G,\sigma)$. In Section $4$, using Lamperti's Theorem we describe the group of isometries of $F_\lambda^p(G,\sigma)$ as the Mackey's group $G_\sigma$, and its homotopy classes (with respect the operator norm topology) as the group $G$. This allows us to prove the first of our rigidity result (Theorem \ref{main-group}). In Section $5$, we define the full twisted group $L^p$-algebra, that is the $L^p$-operator algebra generated by the direct sum of all the twisted representations of the group. We prove that when the group is amenable, then the  full twisted group $L^p$-algebra and the reduced twisted group $L^p$-algebra are isometrically isomorphic regardless of the twist. As an example, we define the noncommutative tori $L^p$-operator algebra as the full twisted $L^p$-operator algebra over $\mathbb{Z}^2$, and we show that two noncommutative tori $L^p$-operator algebras (for $p\in (1,\infty)\setminus \{2\}$) are isometrically isomorphic if and only if the rotations are congruent with respect to the integers. Observe that  the situation is different when $p=2$ \cite{Rieffel1}. In Section $6$, following \cite{RenaultJean1980AGAt,Lupin2014,choi2021rigidity}, given an étale groupoid and a twist (in the sense of \cite{kumjian_1986}), we define the twisted groupoid $L^p$-operator algebra as the completion with respect to the left regular twisted representation of $\Sigma_c(\mathcal{G};\mathcal{E})$, the compact supported $\mathbb{T}$-equivariant functions on the twist $\mathcal{E}$, and will denote it by $F^p_\lambda(\mathcal{G};\mathcal{E})$. By viewing  the twist $\mathcal{E}$ as the linear complex bundle of the groupoid, we can think of $\Sigma_c(\mathcal{G};\mathcal{E})$ as the sections of $\mathcal{E}$. In Section $7$, we will prove the rigidity result for étale groupoids. In order to do so, we will compute the $C^*$-core of $F^p_\lambda(\mathcal{G};\mathcal{E})$, that is the Banach algebras analog of the maximal abelian subalgebras in $C^*$-algebra theory (see \cite{renault2008cartan,kumjian_1986} for example). In contrast with the $C^*$-algebra case, $L^p$-operator algebras (for $p\neq 2$) have a unique $C^*$-core. This is the key result for proving the rigidity result, since this uniqueness of the $C^*$-core makes that any isometric isomorphism between $L^p$-operator algebras reduces to a $C^*$-algebra isomorphism between the $C^*$-cores, and hence one can identify the corresponding Weyl groupoids and the Weyl twists. Finally, if $\mathcal{G}$ is a topologically principal Hausdorff groupoid  and  $\mathcal{E}$ is a twist of  $\mathcal{G}$, then they can be recovered from the  Weyl groupoids and the Weyl twists of $F^p_\lambda(\mathcal{E};\mathcal{G})$, respectively.   This allows us to prove the second of our rigidity result (Theorem \ref{main}).

 \section{Invertible isometries on $L^p$-spaces}
 
 Lamperti's theorem identifies the invertible isometries on $L^p$-spaces when $p\neq2$. In our setting we are working on locally compact groups, for which the left Haar measure is not necessarily a $\sigma$-finite measure. We therefore need  Lamperti's Theorem for the more general case of localizable measure spaces, which is given in \cite{gardella2018isomorphisms}. The concept of a localizable measure is most naturally defined in the setting of measure algebras. We will thus start, following  \cite[Section 2]{gardella2018isomorphisms} and \cite{gardella2019modern}, by defining the measure algebra and measureable function on a measure algebra. For a rigorous and detailed discussion on the  topics see \cite{Fremlin2}.

 Throughout,  let $\mathcal{A}$ be a Boolean algebra.  $\mathcal{A}$ is \emph{($\sigma$-)complete} if every nonempty (countable) subset of $\mathcal{A}$ has a supremum and infimum. A \emph{measure} $\mu$ on $\mathcal{A}$ is map $\mu \colon \mathcal{A} \rightarrow [0,\infty]$ such that $\mu(0) = 0$, $\mu(a) > 0$ for all $a\neq0$ and $\mu(\bigvee_{n\in \mathbb{N}}a_{n}) = \bigvee_{n\in \mathbb{N}}\mu(a_{n})$ whenever $a_{n}$ are pairwise orthogonal. We say that the measure $\mu$ is \emph{semi-finite} if for every $a \in \mathcal{A}$, there exists $b \in \mathcal{A}$ with $b \leq a$ such that $0 < \mu(b) < \infty$. A \emph{measure algebra} is a pair $(\mathcal{A},\mu)$ where $\mathcal{A}$ is a $\sigma$-complete Boolean algebra and $\mu$ is a measure on $\mathcal{A}$. If $\mu$ is semi-finite and $\mathcal{A}$ is complete we say that the measure algebra $(\mathcal{A}, \mu)$ is \emph{localizable}. 
 
 \begin{example}

 	Let $(X,\Sigma,\mu)$ be a measure algebra. The null set $\mathcal{N} = \{E \in \Sigma \colon \mu(E) = 0\}$ is an $\sigma$-ideal in $\Sigma$, and it follows that the quotient $\mathcal{A} = \Sigma / \mathcal{N}$ is a $\sigma$-complete Boolean algebra. Furthermore we can define a natural measure induced by this quotient $\Tilde{\mu} \colon \mathcal{A} \rightarrow [0,\infty]$ given by $\Tilde{\mu}(E + \mathcal{N}) = \mu(E)$ for all $E \in \Sigma$. The induced measure algebra $(\mathcal{A},\Tilde{\mu})$ is called the \emph{measure algebra associated with $(X,\Sigma,\mu)$}.
 \end{example}
 
 Note that the measure space is localizable if and only if the associated measure algebra is localizable \cite[332B]{Fremlin2}. There is also a natural notion of measurable functions and integrals on a measure algebra, as we will define. Let $\mathcal{B}$ be the Borel set of $\mathbb{R}$ (which is also an Boolean algebra). A \emph{measurable real function} on $\mathcal{A}$ is a Boolean homomorphism $f \colon \mathcal{B} \rightarrow \mathcal{A}$ which preserves sumprema of countable sets. We follow the notation of Fremlin \cite{Fremlin2} and write $\{f \in E\}$ for $f(E)$ and $\{f>t\}$ for $f(t,\infty)$. This is to put further emphasis on the connection to the usual measure spaces, as we will see in the following example.
 
 \begin{example}
 	\label{m_func}
 	Let $(X, \Sigma, \mu)$ be a measure space, let $f \colon X \rightarrow \mathbb{R} $ be any measurable real function and let $(\mathcal{A},\bar{\mu})$ be the measure algebra of $(X, \Sigma, \mu)$. We can then obtain a measurable real function on $\mathcal{A}$ from $f$, denoted $\Tilde{f} \colon \mathcal{B} \rightarrow \mathcal{A}$ which is given by $\{\Tilde{f} \in E\} = f^{-1}(E) + \mathcal{N}$ for every Borel subset $E \in \mathcal{B}$. This can also be written as $\{ \Tilde{f} \in E \} = \{x \colon f(x) \in E\} + \mathcal{N}$. 
 \end{example}
 
 %Let $L_{\mathbb{R}}^{0}(\mathcal{A})$ denote the set of all measurable real functions on $\mathcal{A}$. 
 
 %For $f,g\in L_{\mathbb{R}}^{0}(\mathcal{A})$ addition and multiplication is given by 
 %$$
 %\{f + g > t\} = \textrm{sup}_{q\in \mathbb{Q}}\{f>q\}\wedge \{g > t + q\}
 %$$ 
 %and 
 %$$\{fg > t\} =  \textrm{sup}_{q\in \mathbb{Q}}\{f>q\}\wedge \{g > t/q\}$$ respectively, and absolute value is given by 
 %$$\{|f|>t\} = \{f>t\} \vee \{f<t\}.$$
 
 %\begin{example}
 %	Let $(\mathcal{A},\bar{\mu})$ be the measure algebra of some measure space $(X, \Sigma, \mu)$, and let $f,g \colon X \rightarrow %\mathbb{R} $ be any two measurable real functions. We then have that
 %	$$
 %	(f + g)^{-1}(t,\infty) = \bigcup_{q\in Q} \left( f^{-1}(q,\infty) \cap g^{-1}(t-q,\infty) \right).
 %	$$
 %	
 %	Note that that suprema in $(\mathcal{A},\overline{\mu})$ is given by countable union. We therefore see that addition on %$L_{\mathbb{R}}^{0}(\mathcal{A})$ is in line with what we would expect from Example \ref{m_func}. 
 %\end{example}
 
 Let $(\mathcal{A}, \mu)$ be a measure algebra, let $L_{\mathbb{R}}^{0}(\mathcal{A})$ denote the set of all measurable real functions on $\mathcal{A}$, and let $L^{0}(\mathcal{A}) = \{f+ig \colon f,g \in L_{\mathbb{R}}^{0}(\mathcal{A}) \}$ denote the vector space of measurable complex functions on $(\mathcal{A}, \mu)$. Using that $t \mapsto \mu(\{f>t\})$ is a decreasing function, and hence it is Lebesgue measurable, allows us to define the  following norm 
 $$
 \norm{f}_{1} = \int |f| d\mu = \int_{0}^{\infty} \mu(\{|f| > t\})dt\,,
 $$
 for all $f \in L^{0}(\mathcal{A})$. We can then define the space of integrable functions 
 $$
 L^{1}(\mathcal{A}, \mu) := \{f \in L^{0}(\mathcal{A}) : \norm{f}_{1} < \infty \}\,.
 $$
 Let $f \in L^{1}_{\mathbb{R}}(\mathcal{A},\mu)_+$ be a positive real function, then we define the integral
 $$
 \int f d\mu = \int_{0}^{\infty} \mu(\{f > t\})dt \,.
 $$
 
 This extends linearly to $L^{1}(\mathcal{A}, \mu)$ in the usual manner. We then define the $p$-integrable functions for $p \in (1,\infty)$ as follows
 $$
 L^{p}(\mathcal{A}, \mu) := \{f \in L^{0}(\mathcal{A}) : |f|^p \in L^{1}(\mathcal{A}, \mu)\}\,.
 $$
 
 For $p=\infty$, the norm is given by $\norm{f}_{\infty} = \text{inf}\{ t\geq 0 \colon|f|>t\}=0 \}$ and we define 
 $$
 L^{\infty}(\mathcal{A}, \mu) := \{f \in L^{0}(\mathcal{A}) : \norm{f}_{\infty} < \infty \}.
 $$
 Let $(\mathcal{A}, \bar{\mu})$ be the measure algebra of some measure space $(X, \Sigma, \mu)$. From Example \ref{m_func} recall that for every measurable function on $(X, \Sigma, \mu)$ we obtain a measurable function on $(\mathcal{A}, \bar{\mu})$. It turns out that the relation between $(X, \Sigma, \mu)$ and $(\mathcal{A}, \bar{\mu})$ is even stronger. We can actually identify $L^{p}(\mathcal{A}, \bar{\mu})$ with $L^{p}(X, \Sigma, \mu)$ for all $p\in [1,\infty]$ (Corollary 363I and Theorem 366B in \cite{Fremlin2}). This means that we can apply everything we know about $L^{p}(X, \Sigma, \mu)$ to $L^{p}(\mathcal{A}, \bar{\mu})$.
 
 Localizable measure algebras are in fact the largest class for which the Radon Nikodym Theorem is applicable \cite[Theorem 2.7.]{gardella2018isomorphisms}, this allows to prove  Lamperti's Theorem for localizable measure algebras which is usually only proved for $\sigma$-finite measure spaces.

 We define the group of $\mathbb{T}$-valued functions on $\mathcal{A}$ as the following
 
 $$
 \mathcal{U}(L^{\infty}(\mathcal{A})) := \{f \in L^{0}(\mathcal{A}) \colon |f| = 1\}=1\}\,.
 $$
 
 This forms a group under pointwise multiplication. We denote the group of Boolean automorphisms on $\mathcal{A}$ by $\textrm{Aut}(\mathcal{A})$. We denote the group of surjective isometries on $L^{p}(\mathcal{A},\mu)$ by $\textrm{Isom}(L^{p}(\mathcal{A},\mu))$. Following \cite[Section 3]{gardella2018isomorphisms} we define two families of surjective isometries. For a function $f \in \mathcal{U}(L^{\infty}(\mathcal{A}))$,  the map $m_{f} \colon L^{p}(\mathcal{A}, \mu) \rightarrow L^{p}(\mathcal{A}, \mu)$, given by $m_{f}(\xi) = f\xi$ for all $ \xi \in L^{p}(\mathcal{A}, \mu)$ is a surjective isometry, and  for  $\varphi \in \textrm{Aut}(\mathcal{A})$ the map $u_{\varphi} \colon L^{p}(\mathcal{A}, \mu) \rightarrow L^{p}(\mathcal{A}, \mu)$, given by
 $$
 u_{\varphi}(\xi) := (\varphi \circ \xi)
 \left(\frac{d(\mu \circ \varphi^{-1})}{d\mu}\right)^{\frac{1}{p}}
 $$
 is a surjective isometry on $L^{p}(\mathcal{A}, \mu)$. Moreover,	 given another element $\phi \in \textrm{Aut}(\mathcal{A})$, we have that $u_{\varphi} \circ u_{\phi} = u_{\varphi \circ \phi}$.

 And finally  Lamperti's Theorem itself \cite[Theorem 3.7]{gardella2018isomorphisms}.
 
 \begin{theorem}
 	\label{lamperti}
 	Let $(\mathcal{A}, \mu)$ be a localizable measure algebra, let $p \in [1, \infty) \setminus \{ 2 \} $ and let $T \colon L^{p}( \mathcal{A}, \mu ) \rightarrow L^{p}(\mathcal{A}, \mu )$ be a surjective isometry, then there exists a unique $\varphi \in \textrm{Aut}(\mathcal{A})$ and $f \in \mathcal{U}(L^{\infty}(\mathcal{A})) $ such that $T = m_{f}u_{\varphi}$. Moreover, given $f,g \in  \mathcal{U}(L^{\infty}(\mathcal{A})) $ and $\varphi,\phi \in \textrm{Aut}(\mathcal{A})$, we have 
 	$$\norm{m_f u_\varphi - m_g u_\phi}=\max \{ \norm{f-g}_\infty, 2 \delta_{\varphi, \phi}   \}\,.$$ 
 \end{theorem}

 \section{The twisted group $L^p$-operator algebra}
 
 A \emph{Banach algebra} is a Banach space that is also an associative complex algebra such that the norm is submultiplicative. We call the Banach algebra \emph{unital} if it contains a multiplicative identity with norm 1. A \emph{$C^\ast$-algebra} is a Banach algebra $A$, together with an involution such that $\|a^\ast a\| = \|a\|^2$ for all $a \in A$. By the famous Gelfand-Naimark-Segal theorem every $C^*$-algebra can be isometrically embedded into the algebra of bounded operators on a Hilbert space, that is an $L^2$-space. 
 Therefore, one can naturally generalize non-selfadjoint operator algebras in the following way. 
 
 \begin{definition}
 	Let $A$ be a Banach algebra, we say that $A$ is an \emph{$L^p$-operator algebra} if there exists an $L^p$-space $E$ and an isometric homomorphism $A \rightarrow \mathcal{B}(E)$.
 \end{definition}
 %Observe that $L^p$-operator algebras are not in general self-adjoint, even when $p=2$, so 
 Unlike for $C^*$-algebras, it is not known any abstract characterization for $L^p$-operator algebras among all the Banach algebras.

 \begin{definition}
 	\label{cocycle}
 	Let $G$ be a locally compact group. A \emph{continuous 2-cocycle} is a continuous map $\sigma \colon G \times G \rightarrow \mathbb{T}$ satisfying the following:
 	\begin{enumerate}
 		\item $\sigma(x_{1},x_{2})\sigma(x_{1}x_{2},x_{3}) = \sigma(x_{1},x_{2}x_{3})\sigma(x_{2},x_{3})$,
 		\item $\sigma(x,e) = \sigma(e,x) = 1$,
 	\end{enumerate}
 	for all $x, x_{1},x_{2},x_{3} \in G$ and where $e$ is the unit of $G$.
 \end{definition}
 Given a continuous 2-cocycle $\sigma$ we will denote by $\overline{\sigma}$ the continuous 2-cocycle given by $\overline{\sigma}(x,y)=\overline{\sigma(x,y)}$ for every $x,y\in G$.

 For the rest of the paper, $G$ will denote a locally compact group, and the measure will be the left Haar measure. Let $\sigma$ be a continuous 2-cocycle for $G$. The \emph{$\sigma$-twisted convolution product} is defined as
 $$
 (f \ast_\sigma g)(x) = \int_{G} f(y)g(y^{-1}x)\sigma(y,y^{-1}x )dy\,,
 $$
 for $f,g \in L^{1}(G)$. Note that $(L^{1}(G), \ast_\sigma)$ forms a Banach algebra, and we will denote this algebra by $L^{1}(G,\sigma)$. $L^{1}(G,\sigma)$ has an approximate identity (see \cite{Busby}), and it is unital if and only if $G$ is discrete.
 
 We define the  \emph{$\sigma$-twisted left regular representation}  $\lambda_p^\sigma \colon G \rightarrow \mathcal{B}(L^{p}(G))$  as 
 $$
 \lambda_p^\sigma(y)(\xi)(x) = \sigma(y,y^{-1}x)\xi(y^{-1}x)\,,$$ 
 for every $x,y \in G$ and $\xi\in L^{p}(G)$. Similarly, we have the \emph{$\sigma$-twisted right regular representation} $\rho_p^\sigma \colon G \rightarrow \mathcal{B}(L^{p}(G))$ which is given as 
 $$
 \rho_p^\sigma(y)(\xi)(x) = \sigma(x,y)\xi(xy)\,,$$
 for every $x,y \in G$ and $\xi\in L^{p}(G)$.

 \emph{The integrated form of the $\sigma$-twisted left regular representation} $\lambda_p^\sigma \colon L^{1}(G,\sigma) \rightarrow \mathcal{B}(L^{p}(G))$ is defined by 
 $$
 \lambda_p^\sigma(f)(\xi)(x) = \int_{G} f(y)\xi(y^{-1}x)\sigma(y,y^{-1}x )dy\,,$$
 for every $x\in G, f\in L^1(G,\sigma), \xi\in L^p(G)$.
 
 \begin{proposition}
 	\label{commute}
 	Let $G$ be a locally compact group and let $\sigma$ be a continuous 2-cocycle for $G$. Let $p\in [1,\infty)$, then we have that
 	$$
 	\lambda^\sigma_p(z)\rho^{\overline{\sigma}}_p(y) = \rho^{\overline{\sigma}}_p(y)\lambda^\sigma_p(z)\,,
 	$$
 	for all $y,z \in G$.
 \end{proposition}
 \begin{proof}
 	Fix $y,z \in G$, and let $\xi \in L^p(G)$. Then, writing out the expressions we have
 	\begin{align*}
 		(\lambda^\sigma_p(z)\rho^{\overline{\sigma}}_p(y))(\xi)(x) = \sigma(z,z^{-1}x)\overline{\sigma(z^{-1}x,y)}\xi(z^{-1}xy)\,, 
 	\end{align*}
 	for every $x\in G$, and 
 	\begin{align*}
 		(\rho^{\overline{\sigma}}_p(y)\lambda^\sigma_p(z))(\xi)(x) = \overline{\sigma(x,y)}\sigma(z,z^{-1}xy)\xi(z^{-1}xy)\,. 
 	\end{align*}
 	Setting $x_1 = z$, $x_2 = z^{-1}x$ and $x_3 = y$ and using (2) in Defintion \ref{cocycle}, we have that
 	$$
 	\sigma(z,z^{-1}x)\sigma(x,y) = \sigma(z^{-1}x,y)\sigma(z,z^{-1}xy)\,,
 	$$
 	which can be rewritten to 
 	$$
 	\sigma(z,z^{-1}x)\overline{\sigma(z^{-1}x,y)} = \overline{\sigma(x,y)}\sigma(z,z^{-1}xy)\,,
 	$$
 	and the assertion follows.
 \end{proof}
 
 \begin{definition}
 	Let $G$ be a locally compact group and let $\sigma$ be a continuous 2-cocycle on $G$. Let $p \in [1,\infty)$, we define the \emph{reduced $\sigma$-twisted group $L^p$-operator algebra} as the following
 	$$
 	F_\lambda^p(G,\sigma) = \overline{\lambda^\sigma_p(L^1(G,\sigma))}^{\norm{\cdot}} \subseteq \mathcal{B}(L^p(G)).
 	$$
 \end{definition}
 
 We define \emph{$\sigma$-twisted $p$-convolvers}, $CV_p(G,\sigma) = \lambda_p^\sigma(G)^{\prime\prime}$. Note that by Proposition \ref{commute}, the elements in $CV_p(G,\sigma)$ commute with $\rho_p^{\overline{\sigma}}(x)$ for $x \in G$.

 For a locally compact group $G$ and continuous 2-cocycle on $\sigma$ for $G$, we construct the \emph{Mackey group} associated to $G$ and $\sigma$, denoted $G_\sigma$, which as a topological space is the usual product space $\mathbb{T} \times G$, but with product given by
 $$
 (\gamma_{1},x_{1})(\gamma_{2},x_{2}) = (\gamma_{1}\gamma_{2}\sigma(x_{1},x_{2}),x_{1}x_{2})\,,
 $$
 and the inverse given by
 $$
 (\gamma,x)^{-1} = \left(\overline{\gamma \sigma(x^{-1},x)}, x^{-1}\right).
 $$
 
 \begin{definition}
 	Let $G$ be a locally compact group, and let $\sigma, \kappa$ be  two continuous 2-cocycles for $G$, we say that $\sigma$ is \emph{cohomologous} to $\kappa$, denoted 
 	$\sigma \sim \kappa$, if there exists a continuous function $\gamma \colon G \rightarrow \mathbb{T}$ such that $\sigma(x_{1},x_{2})\overline{\kappa(x_{1},x_2)} = \gamma(x_{1})\gamma(x_{2})\overline{
 		\gamma(x_{1}x_{2})}$ for all $x_{1},x_{2} \in G$. \end{definition}

 \section{$L^p$-rigidity for the  reduced $\sigma$-twisted group $L^p$-operator algebra}
 
 In this section we prove one of the main results of this paper (Theorem \ref{main-group}), the rigidity of the twisted group $L^p$-operator algebras. We would like to emphasize that even though most of the results and proofs in this section are extensions of the analogous results in \cite[Section 4]{gardella2018isomorphisms}, we include them because it is not obvious they should work in the case of the  twisted group $L^p$-operator algebras, so all the details must be checked.
 
 Let $G$ be a locally compact group, and let $\sigma$ be a continuous 2-cocycle for $G$. The Banach space of complex-valued Radon measures on $G$ with bounded variation, denoted by $M(G)$, becomes a unital Banach algebra under twisted convolution, where the twisted convolution product for two measures $\nu_{1}, \nu_{2} \in M(G)$ is given by 
 
 $$
 \int_G f(z)d(\nu_{1} \ast_\sigma \nu_{2})(z) := \int_G \int_G f(xy)\sigma(x,y) d\nu_{1}(x)d\nu_{2}(y)\,,
 $$
 for every $f\in C_{c}(G)$. We will denote the algebra by $M(G,\sigma)$ and call it the \emph{$\sigma$-twisted algebra of measures}. 
 Note that for $\nu \in M(G,\sigma)$ the norm can be given by 
 $$\norm{\nu} = \sup_{|f| \leq 1} \left| \int_G f d\nu \right|\,.$$
 Let $f \in C_c(G)$ with $|f| \leq 1$, we then have that
 \begin{align*}
 	\left| \int_G f d (\nu_1 \ast_\sigma \nu_2) \right| &= \left| \int_G \int_G f(xy)\sigma(x,y) d\nu_{1}(x)d\nu_{2}(y) \right| \\ 
 	&\leq \int_G \left|\int_G f(xy)\sigma(x,y) d\nu_{1}(x) \right| d|\nu_{2}|(y) \\
 	&\leq \norm{\nu_1} \norm{\nu_2}.
 \end{align*}
 It follows that $\norm{\nu_1 \ast_\sigma \nu_2} \leq \norm{\nu_1} \norm{\nu_2}$. Let $\nu_1,\nu_2,\nu_3 \in M(G,\sigma)$, we then have  the following:
 \begin{align*}
 	\int_G f(w)d(\nu_{1} \ast_\sigma (\nu_{2} \ast_\sigma \nu_3))(w) &= \int_G \int_G f(xy)\sigma(x,y) d\nu_1(x)(d \nu_2 \ast_\sigma \nu_3)(y)\\
 	&= \int_G \int_G \int_G f(xyz)\sigma(x,yz) d\nu_1(x)\sigma(y,z)d\nu_2(y)d\nu_3(z) \\
 	&= \int_G \int_G \int_G f(xyz)\sigma(x,y)\sigma(xy,z) d\nu_1(x)d\nu_2(y)d\nu_3(z)\\
 	& = \int_G \int_G f(xz)\sigma(x,z) d(\nu_1 \ast_\sigma \nu_2)(x)d\nu_3(z)\\
 	& =\int_G f(w)d((\nu_{1} \ast_\sigma \nu_{2}) \ast_\sigma \nu_3)(w)\,,
 \end{align*}
 where we used (2) in Definition \ref{cocycle} in the third equality. It follows that twisted convolution is associative and submultiplicative. Thus,  $M(G,\sigma)$ is a Banach algebra.

 Now given $p\in [1,\infty)$, we can then define the left regular representation of $M(G,\sigma)$ on $L^p(G)$, denoted by $\lambda_p^\sigma$, which is given by
 $$
 \lambda_p^\sigma(\nu)(f)(x) =
 \int_G \lambda_p^\sigma(y)(f)(x)d\nu(y) = \int_G f(y^{-1}x)\sigma(y,y^{-1}x)d\nu(y)\,,
 $$
 for all $\nu \in M(G,\sigma)$, $f \in L^{p}(G)$ and $x\in G$. This is a bounded operator. Indeed,  given $\xi \in L^{p}(G)$ we have that
 \begin{align*}
 	\norm{\lambda_p^\sigma(\mu) (\xi)}_{p} &= \left( \int \left|\int \lambda_p^\sigma(y)(\xi)(x)d \nu(y) \right|^{p} d\mu(x) \right)^{1/p} \leq \int \left( \int \left|\lambda_p^\sigma(y)(\xi)(x)\right|^{p}d\mu (x)\right)^{1/p} d\nu(y) \\
 	&= \norm{\xi}_p \int_G d\nu \leq \norm{\xi}_p \norm{\mu}_{M^{1}(G)}\,,
 \end{align*}
 where the inequality is due to Minkowski's integral inequality.

 Let $y \in G$ and let $\delta_y$ be the point mass measure associated with $y$. Note that $\lambda_p^\sigma(\delta_y)\xi = \lambda_p^\sigma(y)\xi$ for every $\xi\in L^p(G)$.

 We denote the $L^p$-operator algebra generated by the twisted left representation of the  twisted algebra of measures by $M_{\lambda}^{p}(G,\sigma) := \overline{\lambda_p^\sigma(M(G,\sigma))}^{\norm{\cdot}} \subseteq \mathcal{B}(L^{p}(G))$. 
 
 \begin{definition}
 	Let $A$ be a unital Banach algebra. We define the group of invertible isometries in $A$ by 
 	$$ \mathcal{U}(A)=\{v\in A: v\text{ is invertible and }\norm{v}=\norm{v^{-1}}=1 \}\,.$$
 \end{definition}
 
 \begin{remark}
 	Observe that if $\pi\colon A\to \mathcal{B}(E)$ is a unital isometric representation of $A$ on an $L^p$-space $E$, then $\mathcal{U}(A)$ coincides with $\pi^{-1}(\pi(A)\cap \textrm{Isom}(E))$.
 \end{remark} 
 
 The first goal of this section is to show that we can algebraically identify the Mackey group $G_\sigma$ with the invertible isometries in $CV_{p}(G,\sigma)$, analogous to what  Gardella and Thiel showed in \cite{gardella2018isomorphisms}. Let $(\mathcal{A},\overline{\mu})$ be the Boolean measure algebra associated to $(G,\mu)$.  Fix some element $y \in G$, we write $\sigma_y(x) = \overline{\sigma(x,y)}$, and note that $\sigma_{y} \in \mathcal{U}(L^{\infty}(\mathcal{A}))$. Recall that right multiplication of $y$ induces a Boolean automorphism $r_y \in \textrm{Aut}(\mathcal{A})$ and note that $\rho_p(y) = u_{r_{y^{-1}}} \in \textrm{Isom}(L^p(\mathcal{A}))$. Using notation in Section 2, we therefore have that $\rho_p^{\overline{\sigma}}(y) = m_{\sigma_{y}}u_{r_{y^{-1}}} \in \textrm{Isom}(L^p(\mathcal{A}))$, which will be important for the next result.
 
 \begin{theorem}
 	\label{c}
 	Let $G$ be a locally compact group and let $\sigma$ be a continuous 2-cocycle. Let $p\in[1,\infty)\setminus{\{2\}}$. Then, given $v \in \mathcal{U}(CV_{p}(G,\sigma))$ there exist a unique $\gamma \in \mathbb{T}$ and $g \in  G$ such that $v = \gamma \lambda_p^\sigma(g)$. Furthermore, given $\beta,\gamma\in \mathbb{T}$ and $g,h\in G$ we have
 	\begin{equation}\label{norm_iso}
 		\|\gamma \lambda_p^\sigma(g)-\beta \lambda_p^\sigma(h)\|=\max \{|\beta-\gamma|, 2\delta_{g,h}\}\,.
 	\end{equation}
 \end{theorem}
 \begin{proof}
 	Let $v \in \mathcal{U}(CV_{p}(G,\sigma))$. Then $v \in \textrm{Isom}(L^{p}(\mathcal{A}))$, and by Lamperti's Theorem there exists a unique $h \in \mathcal{U}(L^{\infty}(\mathcal{A}))$ and  $\varphi \in$ Aut($\mathcal{A}$) such that $v = m_{h}u_{\varphi}$. Since every $v \in \mathcal{U}(CV_{p}(G,\sigma))$ commutes with $\rho_p^{\overline{\sigma}}(y)$ for every $y \in G$, we have that $m_{\sigma_{y}}u_{r_{y^{-1}}}m_{h}u_{\varphi} = m_{h}u_{\varphi}m_{\sigma_{y}}u_{r_{y^{-1}}}$. By \cite[Lemma 3.3 \& Lemma 3.4]{gardella2018isomorphisms} we can rewrite the equality to $$
 	m_{\sigma_{y}(r_{y^{-1}} \circ h)}u_{r_{y^{-1}} \circ \varphi} =  m_{h(\varphi \circ \sigma_{y})}u_{\varphi \circ r_{y^{-1}}}.
 	$$
 	
 	This implies that $r_{y^{-1}} \circ \varphi = \varphi \circ r_{y^{-1}}$ and $\sigma_{y}(r_{y^{-1}} \circ h) = h(\varphi \circ \sigma_{y})$. From the first equality, it follows by \cite[Lemma 4.8]{gardella2018isomorphisms} that there exists a unique $g_v \in G$ such that $\varphi = l_{g_v}$. We then have $h(l_{g_v} \circ \sigma_y) = \sigma_y(r_{y^{-1}} \circ h)$. This implies that $h(x)\overline{\sigma(g_v^{-1}x,y)} = \overline{\sigma(x,y)}h(xy)$ for all $x \in G$. Set $x=e$, then we deduce that 
 	$$
 	h(y) = h(e)\overline{\sigma(g_v^{-1},y)} = h(e)\overline{\sigma(g_v,g_v^{-1})}\sigma(g_v,g_v^{-1}y)\,.
 	$$
 	Now set $\gamma_{v} =h(e)\overline{\sigma(g_v^{-1},g_v)}$ and note that $\gamma_{v} \in \mathbb{T}$. Then given $\xi\in L^p(G)$ we have that $v(\xi)(x)  = \gamma_{v}\sigma(g_v,g_v^{-1}x)\xi(g_v^{-1}x) = \gamma_{v}\lambda_p^\sigma(g_v)(\xi)(x)$ for every $x\in G$, as desired. Finally, (\ref{norm_iso}) follows from \cite[Theorem 3.7]{gardella2018isomorphisms}.
 \end{proof}
 
 Given a unital Banach algebra $A$, let $\mathcal{U}(A)_0$ denote the connected component of $\mathcal{U}(A)$ in the norm topology that contains the unit of $A$. This is then a normal subgroup of $\mathcal{U}(A)$ and we write 
 $$
 \pi_0(\mathcal{U}(A)) = \mathcal{U}(A) /  \mathcal{U}(A)_0
 $$
 for the quotient.
 
 \begin{proposition}
 	Let $G$ be a locally compact group and let $\sigma$ be a continuous 2-cocycle. For $p\in[1,\infty)\setminus{\{2\}}$, there is a natural group isomorphism $G_\sigma \cong \mathcal{U}(CV_{p}(G,\sigma))$ and $G \cong \pi_{0}(\mathcal{U}(CV_{p}(G,\sigma)))$ (as discrete groups)  given by the maps $(\gamma,y) \mapsto \gamma \lambda_p^\sigma(y)$ and $y \mapsto [\lambda_p^\sigma(y)]$ respectively.
 \end{proposition}
 \begin{proof}
 	Define $\Delta \colon \mathcal{U}(CV_{p}(G,\sigma)) \rightarrow G_{\sigma}$ by $\Delta(v) = (\gamma_{v},g_{v})$ where $\gamma_{v}\in \mathbb{T}$ and $g_v\in G$ are given in the proof of Theorem \ref{c}. The map is injective, and it is surjective as $\gamma\lambda_p^\sigma(y)$ is in $CV_p(G,\sigma)$ for all $\gamma \in \mathbb{T}$ and all $y \in G$. It remains to show that this is a homomorphism. Let $\gamma_{1}, \gamma_{2} \in \mathbb{T}$, and $y_{1},y_{2} \in G  $, then
 	\begin{align*}
 		(\gamma_{1}\lambda_p^\sigma(y_{1}) \circ \gamma_{2}\lambda_p^\sigma(y_{2}))(\xi)(x) &= (\gamma_{1}\lambda_p^\sigma(y_{1}))(\gamma_{2}\lambda_p^\sigma(y_{2})(\xi))(x) \\
 		& = \gamma_1\sigma(y_1,y_1^{-1}x)(\gamma_{2}\lambda_p^\sigma(y_{2})(\xi))(y_1^{-1}x)\\
 		& = \gamma_1\gamma_2\sigma(y_1,y_1^{-1}x)\sigma(y_2,y_2^{-1}y_1^{-1}x)\xi(y_2^{-1}y_1^{-1}x) \\
 		& =\gamma_1\gamma_2 \sigma(y_1,y_2)\sigma(y_1y_2,(y_1y_2)^{-1}x)\xi(y_2^{-1}y_1^{-1}x) \\
 		& = \gamma_1\gamma_2\sigma(y_1,y_2) \lambda_p^\sigma(y_1y_2)(\xi)(x)\,,
 	\end{align*}
 	where we used (1) in Definition \ref{cocycle} with $x_1 = y_1$, $x_{2} = y_{2}$ and $x_3 = (y_{1}y_{2})^{-1}x$. Thus, 
 	$$
 	\Delta(\gamma_{1}\lambda_p^\sigma(y_{1}) \circ \gamma_{2}\lambda_p^\sigma(y_{2})) = (\gamma_{1}\gamma_{2}\sigma(y_{1},y_{2}),y_{1}y_{2}) = (\gamma_{1},y_{1})(\gamma_{2},y_{2}) = \Delta(\gamma_{1}\lambda_p^\sigma(y_{1}))\Delta(\gamma_{2}\lambda_p^\sigma(y_{2}))\,.
 	$$
 	Hence $\Delta$ is a homomorphism and the assertion follows. Finally, because of (\ref{norm_iso}) we have that $\Delta$ induces a group isomorphism between $G$ and $\pi_{0}(\mathcal{U}(CV_{p}(G,\sigma)))$.
 \end{proof}
 
 \begin{lemma}
 	\label{incl}
 	Let $G$ be a locally compact group and let $\sigma$ be a continuous 2-cocycle. Let $(f_j)_j$ be a contractive approximate identity in $L^{1}(G,\sigma)$. For $p\in(1,\infty)\setminus{\{2\}}$,  then the net $(\lambda^\sigma_{p}(f_{j}))_j$ converges weak$^{\ast}$ to 1 in $\mathcal{B}(L^{p}(G))$.
 \end{lemma}
 \begin{proof}
 	The proof follows  \cite[Lemma 4.2]{gardella2018isomorphisms}. Let $q$ be the conjugate exponent of $p$. We can we identify $\mathcal{B}(L^{p}(G))$ with the dual of $L^{p}(G) \widehat{\otimes} L^{q}(G)$, via the dual pairing $(T,f \otimes g) = (T(f),g)$, where $T \in \mathcal{B}(L^{p}(G))$, $f \in L^p(G),g \in L^q(G)$, see  \cite[Section 2]{Daws_2019}. Now, since the net is bounded, it suffices to show that $(\lambda_p^\sigma(f_{j}), \omega) \rightarrow (1,\omega)$ where $\omega$ is a simple tensor in $L^{p}(G) \widehat{\otimes} L^{q}(G)$. Let $\xi \in L^{p}(G)$ and $\eta \in L^{q}(G)$. By assumption we have that $(f_j)_j$ is a contractive approximate identity, and  we claim that $\norm{f_j \ast_\sigma \xi - \xi}_p \rightarrow 0$. Indeed, using Minkowski's inequality  we have that
 	\begin{align*}
 		\norm{f_j \ast_\sigma \xi - \xi}_p & = \left( \int \left|  \int (\sigma(y,y^{-1}x)\xi(y^{-1}x)-\xi(x))f_j(y)\,dy\right|^p\,dx \right)^{1/p} \\
 		& \leq  \int \left( \int \left| (\sigma(y,y^{-1}x)\xi(y^{-1}x)-\xi(x))f_j(y)\right|^p\,dx\right)^{1/p} \,dy \\
 		& = \int \left( \int \left| (\sigma(y,y^{-1}x)\xi(y^{-1}x)-\xi(x))\right|^p\,dx\right)^{1/p}|f_j(y)| \,dy \,.
 	\end{align*}
 	Now given any $U\subseteq G$ we have that 
 	\begin{align*}
 		\norm{f_j \ast_\sigma \xi - \xi}_p & \leq \int_U \left( \int \left| (\sigma(y,y^{-1}x)\xi(y^{-1}x)-\xi(x))\right|^p\,dx\right)^{1/p}|f_j(y)| \,dy \\ 
 		& + \int_{G\setminus U} \left( \int \left| (\sigma(y,y^{-1}x)\xi(y^{-1}x)-\xi(x))\right|^p\,dx\right)^{1/p}|f_j(y)| \,dy \\
 		& \leq \int_U \left( \int \left| (\sigma(y,y^{-1}x)\xi(y^{-1}x)-\xi(x))\right|^p\,dx\right)^{1/p}|f_j(y)| \,dy  + 2\|\xi\|_p\int_{G\setminus U} |f_j(y)|\,dy\,.
 	\end{align*}
 	
 	But now taking $U$ a small neighborhood of the unit of $G$ and big enough $j$ we can make $\norm{f_j \ast_\sigma \xi - \xi}_p$ as small as we want. Thus,  $\norm{f_j \ast_\sigma \xi - \xi}_p \rightarrow 0$ as desired. 
 	
 	Then we have that  
 	$$
 	(\lambda^\sigma_p(f_{j}),\xi \otimes \eta) = (f_j \ast_\sigma \xi, \eta) \rightarrow (\xi,\eta) = (1,\xi \otimes \eta)
 	$$
 	and the statement follows.
 \end{proof}
 
 Let $A$ be a unital Banach algebra, we define the \emph{left multiplier algebra of $A$} as $$M_l(A) := \{S \in \mathcal{B}(A) \colon S(ab) = S(a)b \textrm{ for all } a,b \in A\}\,.$$
 
 Since $F^p_\lambda(G,\sigma)$ has an approximate identity and $F^p_\lambda(G,\sigma) \subseteq \mathcal{B}(L^{p}(G))$ is a non degenerate subalgebra, it follows from \cite[Theorem 4.1]{gardella2018extending} that $M_{l}(F^p_\lambda(G,\sigma))$ has a canonical isometric representation as a unital subalgebra in $\mathcal{B}(L^{p}(G))$.
 
 \begin{proposition}
 	\label{c-identify}
 	Let $G$ be a locally compact group and let $\sigma$ be a continuous 2-cocycle. For $p\in(1,\infty)\setminus{\{2\}}$, then there is a natural isometric inclusion
 	$$
 	M_{\lambda}^{p}(G,\sigma) \subseteq  M_{l}(F^p_\lambda(G,\sigma)) \subseteq CV_{p}(G,\sigma).
 	$$
 \end{proposition}
 
 \begin{proof}
 	We will first show the first inclusion. It suffices to show that $\lambda_p^\sigma(M(G,\sigma)) \subseteq M_{l}(F^p_\lambda(G,\sigma))$ since the latter is closed by the operator norm. Let $\mu \in M(G,\sigma)$ and let $a \in F^p_\lambda(G,\sigma)$. Then there exists a net 
 	$(f_{j})$ in $L^{1}(G,\sigma)$ with $(\lambda^\sigma_p(f_j))$ converging to $a$ in the operator norm. We have that $\mu \ast_\sigma f_{j} \in L^{1}(G)$, and it follows from this that $\lambda_p^\sigma(\mu \ast_\sigma f_{j}) \in  F^p_\lambda(G,\sigma)$. Thus $\lambda_p^\sigma(\mu \ast_\sigma f_{j}) = \lambda_p^\sigma(\mu) \lambda_p^\sigma(f_{j}) \rightarrow \lambda_p^\sigma(\mu)a$ in the operator norm, and the inclusion follows.
 	
 	We will then show that $M_{l}(F^p_\lambda(G,\sigma)) \subseteq CV_{p}(G,\sigma)$. Let $S \in M_{l}(F^p_\lambda(G,\sigma))$, and let $(f_{j})_{j}$ be a contractive approximate identity in $L^{1}(G,\sigma)$. Note that for every $j$, $S\lambda_p^\sigma(f_{j}) \in F^p_\lambda(G,\sigma)$. By Lemma \ref{incl}  $S\lambda_p^\sigma(f_{j}) \xrightarrow{w^{\ast}} S$ and thus $M_{l}(F^p_\lambda(G))$ is in the weak$^{\ast}$ closure of $F^p_\lambda(G)$. Since $\lambda_p^\sigma(L^{1}(G,\sigma)) \subseteq CV_{p}(G,\sigma)$ and the fact that the commutant algebras are weak operator closed, and hence $w^{\ast}$-closed, we deduce that $M_{l}(F^p_\lambda(G,\sigma)) \subseteq CV_{p}(G,\sigma)$.
 \end{proof}
 
 \begin{proposition}
 	
 	\label{c-homo}
 	Let $G$ be a  locally compact group and let $\sigma$ be a continuous 2-cocycle. For $p\in(1,\infty)\setminus{\{2\}}$, let $A$ be a normed closed subalgebra of $\mathcal{B}(L^{p}(G))$ such that $M_{\lambda}^{p}(G,\sigma) \subseteq A \subseteq CV_{p}(G,\sigma)$. Then there are natural group isomorphisms
 	$$
 	\Delta \colon G_\sigma \xrightarrow{\cong} \mathcal{U}(A) \quad \text{and}\quad\Delta^{\prime} \colon G \xrightarrow{\cong} \pi_0(\mathcal{U}(A))
 	$$
 	given by $(\gamma,y) \mapsto \gamma\lambda_p^\sigma(y)$ and  $y \mapsto [\lambda_{p}(y)]$, respectively.
 \end{proposition}
 
 \begin{proof}
 	The map $\Delta\colon G_\sigma \xrightarrow{\cong} \mathcal{U}(CV_{p}(G,\sigma))$ is a group isomorphism by Lemma \ref{c}. It is enough to prove  that  $\mathcal{U}(M_{\lambda}^{p}(G,\sigma)) =  \mathcal{U}(CV_{p}(G,\sigma))$. That $\mathcal{U}(M_{\lambda}^{p}(G,\sigma)) \subseteqq  \mathcal{U}(CV_{p}(G,\sigma))$ is clear. On the other side, let $v$ be any element in $ \mathcal{U}(CV_{p}(G,\sigma))$, then by Lemma \ref{c} we get that $v = \gamma \lambda_p^\sigma(y)$ for some $\gamma \in \mathbb{T}$ and $y \in G$. But since $\lambda_p^\sigma(y) = \lambda_p^\sigma(\delta_{y})$ we get that $v \in M_{\lambda}^{p}(G,\sigma)$, and the statement follows.
 \end{proof}
 
 To recover the topology of the group, we need to briefly  mention some topological concepts. Let $E$ be a Banach space and let $(T_j)_j$ be a net in $\mathcal{B}(E)$. We say that $(T_j)_j$ converges to $T$ in the \emph{strong operator topology} (SOT) of $\mathcal{B}(E)$ if and only if $T_j(\xi)$ converges to $T(\xi)$ in the norm of $E$ for every $\xi\in E$. 
 
 Let $A$ be a unital Banach algebra. We  need to define the strict operator topology on $M_l(A)$ which is the restriction of the strong operator topology of $\mathcal{B}(A)$ to $M_l(A)$. In other words, for a net $(S_j)_j$ in $M_l(A)$, we have that $S_j \xrightarrow{str} S$ for some $S \in M_l(A)$ if and only if $S_j(a) \xrightarrow{\norm{\cdot}} S(a)$ for all $a \in A$. We let $\mathcal{U}(M_l(A))_{str}$ denote the group $\mathcal{U}(M_l(A))$ with the strict topology on $M_l(A)$ restricted to the invertible isometries, then  the  subgroup of the invertible isometries homotopic to the identity $\mathcal{U}(M_l(A))_0$ is a closed subgroup, and we denote by $\pi_0(\mathcal{U}(M_l(A)))_{str}$ for the group $\pi_0(\mathcal{U}(M_l(A)))$ endowed with the quotient topology.
 
 \begin{proposition}
 	\label{sigmaleft_mult}
 	Let $G$ be a locally compact group  and let $\sigma$ be a continuous 2-cocycle. For $p \in (1,\infty)\setminus\{2\}$,
 	there are natural isomorphisms as topological groups
 	$$
 	\Lambda \colon G_\sigma \rightarrow \mathcal{U}(M_{l}(F^p_\lambda(G,\sigma)))_{str} \quad\text{and}\quad \Lambda^{\prime} \colon G \rightarrow \pi_{0}(\mathcal{U}(M_{l}(F^p_\lambda(G,\sigma))))_{str}\,,
 	$$
 	given by $\Lambda(\gamma,x) = \gamma\lambda_p^\sigma(x)$ and $\Lambda^{\prime}(x) = [\lambda_{p}(x)]$, respectively.
 \end{proposition}
 
 \begin{proof}
 	We have the following commutative diagram:
 	%\[ \begin{tikzcd}
 		%\mathbb{T} \times  G \arrow{r}{\Lambda} \arrow[swap]{d}{(\gamma,y)\mapsto y} & \mathcal{U}(M_{l}(F^p_\lambda(G,\sigma)))_{str} %\arrow{d}{v \mapsto [v]} \\%
 		%G \arrow{r}{\Lambda}& \pi_{0}(\mathcal{U}(M_{l}((F^p_\lambda(G,\sigma)))_{str}.
 		%\end{tikzcd}
 		%\]
 		$$\xymatrix {  {G_\sigma}\ar@{>}[rr]^\Lambda\ar@{>}[d]^{(\gamma,x)\mapsto x}& \qquad &  {\mathcal{U}(M_{l}(F^p_\lambda(G,\sigma)))_{str}}\ar@{>}[d]_{v \mapsto [v]}\\
 			{G}\ar@{>}[rr]^{\Lambda'} & \qquad&{\pi_{0}(\mathcal{U}(M_{l}((F^p_\lambda(G,\sigma)))_{str}} }\,.$$
 		By Proposition \ref{c-homo} and Proposition \ref{c-identify} it follows that $\Lambda$ and $\Lambda^{\prime}$ are group isomorphisms (as discrete groups). Since the downward maps in the diagram  are quotient maps it is suffices to show that $\Lambda$ is a homeomorphism. Let $(\gamma_{j}, x_{j})_j$  be a net in $G_\sigma$ converging to some element $(\gamma, x) \in G_\sigma$. Given $f \in L^{1}(G,\sigma)$ we have that $(\delta_{x_{j}} \ast_\sigma f) \xrightarrow{\norm{\cdot}_1} (\delta_x \ast_\sigma f)$ in $L^{1}(G,\sigma)$. Since $\lambda_p^\sigma$ is a contractive homomorphism, this implies that 
 		$\gamma_{j} \lambda_p^\sigma(x_{j})\lambda_p^\sigma(f) = \gamma_{j}\lambda_p^\sigma(\delta_{x_{j}} \ast_\sigma f) \xrightarrow{\norm{\cdot}} \gamma\lambda_p^\sigma(\delta_{x} \ast_\sigma f) = \gamma \lambda_p^\sigma(x)\lambda_p^\sigma(f)$ in $F^p_\lambda(G, \sigma)$ for every $f \in L^{1}(G,\sigma)$. Since $F_\lambda^p(G,\sigma)$ is the closure of $\lambda_p^\sigma(L^1(G,\sigma))$ and the net $\left(\gamma_{j}\lambda_p^\sigma(x_{j})\right)_{j}$ is bounded, we deduce that $\gamma_{j} \lambda_p^\sigma(x_{j})a \xrightarrow{\norm{\cdot}} \gamma \lambda_p^\sigma(x)a$ for all $a \in F^p_\lambda(G,\sigma)$, and so we have, by the definition of the strict topology, that $\gamma_{j} \lambda_p^\sigma(x_{j}) \xrightarrow{str} \gamma \lambda_p^\sigma(x)$ in $M_{l}(F^p_\lambda(G,\sigma))$. 
 		
 		Conversely assume that $\gamma_{j} \lambda_p^\sigma(x_{j}) \xrightarrow{str} \gamma \lambda_p^\sigma(x)$ in $M_{l}(F^p_\lambda(G,\sigma))$. We need to show that $(\gamma_{j}, x_{j}) \rightarrow (\gamma,x)$ in $G_\sigma$, which means we have to show that $\gamma_{j} \rightarrow \gamma$  in $\mathbb{T}$ and $x_{j} \rightarrow x$ in $G$. 
 		
 		We will first show that $x_{j} \rightarrow x$. Let $U$ be a neighborhood of the unit of $G$. Choose a neighborhood $V$ also containing the unit of $G$ such that $V^{2}V^{-2} \subseteq U$ and $\mu(V) < \infty$. This means that $\chi_{V} \in L^{p}(G)$ for every $p \in [1,\infty)$. Note that this implies that $\lambda_p^\sigma(\chi_V) \in F^p_\lambda(G,\sigma)$ and it follows by the definition of the strict topology that
 		$\gamma_{j} \lambda_p^\sigma(x_{j})\lambda_p^\sigma(\chi_{V}) \xrightarrow{\norm{\cdot}} \gamma \lambda_p^\sigma(x)\lambda_p^\sigma(\chi_{V})$ in $F_\lambda^p(G,\sigma)$ 
 		and so we have $\gamma_{j} \lambda_p^\sigma(x_{j})\lambda_p^\sigma(\chi_{V})\chi_{V} \xrightarrow{\norm{\cdot}} \gamma \lambda_p^\sigma(x)\lambda_p^\sigma(\chi_{V})\chi_{V}$ in $L^{p}(G)$. Observe that 
 		$$
 		f(z):=\lambda_p^\sigma(\chi_{V})(\chi_{V})(z) = \int_G \sigma(t,t^{-1}z)\chi_V(t)\chi_{V}(t^{-1}z)dt 
 		\,,$$
 		defines a function in $L^p(G)$ with  $\textrm{supp}(f) \subseteq V^{2}$. 
 		
 		If the net converges to $x$, then every open neighborhood of $x$ will contain the net eventually. We will assume that this is not the case, and show that this leads to a contradiction. Since $U$ is an arbitrary neighborhood of the unit we assume that $x_{j} \notin xU$ for all $j$, then $x_{j} \notin xV^2V^{-2}$, and thus $x_{j}V^{2} \cap xV^{2} = Ø$, which implies that $l_{x_j}(f)$ and $l_x(f)$ have disjoint support, which again implies that $\sigma(x_j,x_j^{-1}z)f(x_{j}^{-1}z)$ and $\sigma(x,x^{-1}z)f(x^{-1}z)$ have disjoint supports. Using these facts we have that
 		$$
 		\norm{\gamma_{j}\sigma(x_j,x_j^{-1}z)f(x_{j}^{-1}z) - \gamma \sigma(x,x^{-1}z)f(x^{-1}z)}^p_{p}  =  \int_{x_{j}V^{2}}\left|f(x_{j}^{-1}z)\right|^{p}dz + \int_{sV^{2}}\left|f(x^{-1}z)\right|^{p}dz   > 0.
 		$$
 		for all $j$. But we have already shown that 
 		$$\gamma_{j}\sigma(x_j,x_j^{-1}z)f(x_{j}^{-1}z) =\gamma_{j} \lambda_p^\sigma(x_{j})\lambda_p^\sigma(\chi_{V})\chi_{V}(z) \xrightarrow{\norm{\cdot}} \gamma \lambda_p^\sigma(x)\lambda_p^\sigma(\chi_{V})\chi_{V}(z) = \gamma \sigma(x,x^{-1}z)f(x^{-1}z),
 		$$
 		and so we have a contradiction. Thus, for every neighborhood $U$ containing the unit, $x_{j}$ is eventually in $xU$ for large enough $j$, and hence $x_{j} \rightarrow x$. 
 		
 		Finally it remains to show that $\gamma_{j} \rightarrow \gamma$. We have that 
 		
 		\begin{align*}\norm{\gamma_i \lambda_p^\sigma(x)\lambda_p^\sigma(\chi_V)-\gamma \lambda_p^\sigma(x)\lambda_p^\sigma(\chi_V)} & \leq \norm{\gamma_i \lambda_p^\sigma(x)\lambda_p^\sigma(\chi_V)-\gamma_i \lambda_p^\sigma(x_i)\lambda_p^\sigma(\chi_V)} \\ &+\norm{\gamma_i \lambda_p^\sigma(x_i)\lambda_p^\sigma(\chi_V)-\gamma \lambda_p^\sigma(x)\lambda_p^\sigma(\chi_V)}\,,\end{align*}
 		but then since both terms in the right hand side go to $0$ so does the left hand side. Then $\norm{\gamma_i \lambda_p^\sigma(x)\lambda_p^\sigma(\chi_V)-\gamma \lambda_p^\sigma(x)\lambda_p^\sigma(\chi_V)}=|\gamma_i-\gamma|\norm{\lambda_p^\sigma(x)\lambda_p^\sigma(\chi_V)}\to 0$, so $\gamma_i\to \gamma$, as desired. 
 	\end{proof}
 	
 	\begin{theorem}
 		\label{main-group}
 		Let $G$ and $H$ be two locally compact groups and let $\sigma$ be a continuous 2-cocycle for $G$ and $\kappa$  be a continuous 2-cocycle for $H$. For $p \in (1,\infty)\setminus\{2\}$, then there exists an isometric isomorphism between  $F^p_\lambda(G,\sigma)$ and $F^p_\lambda(H,\kappa)$  if and only if $G \cong H$ as topological groups and $\sigma \sim \kappa$. 
 	\end{theorem}
 	\begin{proof}
 		Assume that there is an isometric isomorphism
 		$F^p_\lambda(G,\sigma) \rightarrow F^p_\lambda(H,\kappa)$,
 		this induces an isometric isomorphism 
 		$$\Phi \colon M_{l}(F^p_\lambda(G, \sigma)) \rightarrow M_{l}(F^p_\lambda(H, \kappa))\,.$$
 		Note that $\Phi$ and its inverse are norm continuous and thus strictly continuous, then it follows that $\Phi$ is a homeomorphism with respect to the strict topology between  $M_{l}(F^p_\lambda(G,\sigma))$ and $M_{l}(F^p_\lambda(H, \kappa))$. We will for the rest of the proof write $A=M_{l}(F^p_\lambda(G,\sigma))$ and $B=M_{l}(F^p_\lambda(H,\kappa))$. Then $\Phi$ induces an isomorphism of the topological groups
 		\begin{align*}
 			\phi \colon \mathcal{U}(A)_{str} &\rightarrow \mathcal{U}(B)_{str} \\
 		\end{align*}
 		and it follows that
 		\begin{align*}
 			\phi^{\prime} \colon \pi_{0}(\mathcal{U}(A))_{str} &\rightarrow \pi_{0}(\mathcal{U}(B))_{str} \\ 
 		\end{align*}
 		is an isomorphism of topological groups. By Proposition \ref{sigmaleft_mult} we have the following isomorphisms of topological groups 
 		\begin{align*}
 			\Lambda_G \colon G_\sigma \xrightarrow{\cong} \mathcal{U}(A)
 			\textrm{   and   }  \Lambda_H \colon H_\rho \xrightarrow{\cong} \mathcal{U}(B)\,.
 		\end{align*}
 		We thus have the following commuting diagram
 		\begin{equation}
 			\label{diagram1}
 			%\begin{tikzcd}
 			%G_{\sigma} \arrow{r}{\Lambda_G} \arrow[swap]{d}{(\gamma,g)\mapsto g} & \mathcal{U}(A)_{str} \arrow{d}{v \mapsto [v]} \arrow{r}{\phi} & %\mathcal{U}(B)_{str} \arrow{d}{u \mapsto [u]} &
 			%\arrow{l}{\Lambda_H} H_{\rho} \arrow{d}{(\gamma,h)\mapsto h} \\
 			%G \arrow{r}{\Lambda^{\prime}_G} &
 			%\pi_{0}(\mathcal{U}(A)_{str} \arrow{r}{\phi^{\prime}} & \pi_{0}(\mathcal{U}(B)_{str}  & 
 			%\arrow{l}{\Lambda^{\prime}_{H}} H
 			%\end{tikzcd}
 			\xymatrix { G_{\sigma}\ar@{>}[rr]^{\Lambda_G}\ar@{>}[d]^{(\gamma,g)\mapsto g} & \qquad &  \mathcal{U}(A)_{str}\ar@{>}[d]^{v \mapsto [v]}\ar@{>}[r]^{\phi} & \mathcal{U}(B)_{str}\ar@{>}[d]^{u \mapsto [u]} & \qquad & H_{\rho}\ar@{>}[ll]_{\Lambda_H}\ar@{>}[d]^{(\gamma,h)\mapsto h} \\ 
 				G \ar@{>}[rr]^{\Lambda^{\prime}_G} & \qquad & \pi_{0}(\mathcal{U}(A))_{str}\ar@{>}[r]^{\phi^{\prime}} &  \pi_{0}(\mathcal{U}(B))_{str} & \qquad  & H\ar@{>}[ll]_{\Lambda^{\prime}_{H}} }
 		\end{equation}
 		
 		where the horizontal maps are isomorphisms of topological groups and the downward maps are quotient maps. It follows that $G \cong H$ as topological groups. 
 		
 		It remains to show $\kappa \sim \sigma$. We have shown that $G_\sigma \cong H_\kappa$. Denote this isomorphism by $\varphi \colon G_\sigma \rightarrow H_\kappa$. Since $\varphi$ is an isomorphism, for every $g \in G$, we have that $\varphi(1,g) = (\gamma_{g},h_{g})$ for some unique $\gamma_g \in \mathbb{T}$ and some $h_g \in H$. This induces a continuous map $\gamma \colon G \rightarrow \mathbb{T}$ given by  $\gamma(g): =\gamma_g$, and a map $h \colon G \rightarrow H$ given by $h(g) := h_g$. Observe that $h$ is a continuous injective homomorphism. By the commutative diagram (\ref{diagram1}), we also have that $\varphi(\lambda,e_G) = (\lambda^\prime,e_H)$ for all $\lambda \in \mathbb{T}$ where $e_G$ its the unit of $G$, $e_H$ is the unit in $H$ and $\lambda^\prime$ is some element in $\mathbb{T}$. This also induces a map $\pi\colon \mathbb{T} \rightarrow \mathbb{T}$ given by $\pi(\lambda) = \lambda^\prime$, this is a continuous injective homomorphism of the unit circle, and since $\Phi$ is linear we deduce that $\lambda^\prime = \lambda$. Now let $g_{1}, g_{2} \in G$, then
 		\begin{align*}
 			(\gamma(g_1 g_2), h(g_1g_2)) &= \varphi((1,g_{1}g_{2})) \\
 			&= \varphi((\overline{\sigma(g_{1},g_{2})},e_G)(1,g_{1})(1,g_{2})) \\
 			&= (\overline{\sigma(g_{1},g_{2})},e_H)\varphi(1,g_{1})\varphi(1,g_{2}) \\ 
 			&= (\gamma(g_{1})\gamma(g_{2})\overline{\sigma(g_{1},g_{2})},e_H)(1,h_{g_1})(1,h_{g_2}) \\
 			&= (\gamma(g_{1})\gamma(g_{2})\overline{\sigma(g_{1},g_{2})},e_H)(\kappa(h_{g_1}, h_{g_2}), h_{g_1} h_{g_2}) \\
 			&= (\gamma(g_{1})\gamma(g_{2})\overline{\sigma(g_{1},g_{2})}\kappa(h_{g_1}, h_{g_2}),h_{g_1} h_{g_2})
 		\end{align*}
 		We thus have that $\gamma(g_1 g_2) = \gamma(g_{1})\gamma(g_{2})\overline{\sigma(g_{1},g_{2})}\kappa(g_1,g_2)$. Rewriting the expression we deduce that $\sigma(g_{1},g_{2})\overline{\kappa(g_1,g_2)} = \gamma(g_{1})\gamma(g_{2})\overline{\gamma(g_1 g_2)}$ which implies that $\sigma \sim \kappa$.
 		
 		Conversely assume that $G \cong H$, we can without loss of generality assume $G = H$ and assume that  $\sigma \sim \kappa$, then there exists a continuous map $\gamma \colon G \rightarrow \mathbb{T}$ such that
 		$\sigma(x_1,x_2)\overline{\kappa(x_1,x_2)} = \gamma(x_1)\gamma(x_2)\overline{\gamma(x_1x_2)}$ for every $x_1,x_2 \in G$. Define the map $\phi \colon L^{1}(G,\sigma) \rightarrow L^{1}(G,\kappa)$ by $\phi(f(x)) = \gamma(x)f(x)$. It is easy to see that this is a surjective isometry, and we are going to show that this is also a homomorphism. First note that $\sigma(y,y^{-1}x)\gamma(x) = \kappa(y,y^{-1}x)\gamma(y)\gamma(y^{-1}x)$. We then have that
 		\begin{align*}
 			(\phi(f_1)\ast_{\kappa} \phi(f_2))(x) &= (\gamma f_1 \ast_{\kappa} \gamma f_2)(x)\\
 			&= \int_G f_1(y)\gamma(y)\gamma(y^{-1}x)\kappa(y,y^{-1}x)f_2(y^{-1}x)dy \\
 			&= \gamma(x)(f_1 \ast_{\sigma} f_2)(x) = \phi((f_1 \ast_{\sigma} f_2))(x)\,,
 		\end{align*}
 		for every $x\in G$. So $\phi$ is an isometric isomorphism. This induces an isometric isomorphism 
 		$$
 		\Tilde{\phi} \colon (L^{1}(G,\sigma),\norm{\cdot}_{F^p_\lambda(G,\sigma)}) \rightarrow (L^{1}(G,\kappa), \norm{\cdot}_{F^p_\lambda(H,\kappa)}),
 		$$
 		which extends to the closures, thus $F^p_\lambda(G, \sigma)$ is isometrically isometric to $F^p_\lambda (G, \kappa)$.
 	\end{proof}
 	
 	For $p=2$,  $F^{2}_\lambda(G,\sigma)$ is the reduced group $C^\ast$-algebra usually denoted by $C_{\lambda}^*(G,\sigma)$. It is well known that whenever $G\cong H$ and $\sigma\sim \kappa$ then $C^*_\lambda(G,\sigma)\cong C^*_\lambda(H,\kappa)$, however the other direction it is not true in general, for example $C^{\ast}_{\lambda}(\mathbb{Z}_{2}\times\mathbb{Z}_{2},1)$ is isometrically isomorphic to $C^{\ast}_{\lambda}(\mathbb{Z}_{4},1)$ where $1$ denotes the trivial twist, but the groups $\mathbb{Z}_{2}\times\mathbb{Z}_{2}$ and $\mathbb{Z}_{4}$ are not isomorphic. 
 	
 	\section{Twisted $L^p$-operator algebras of amenable groups.}

 	For the trivial twist, Gardella and Thiel showed in \cite[Theorem 3.7]{Gardella_2015} that $F^p(G)$ and  $F^p_\lambda(G)$ are isometrically isomorphic if and only if $G$ is amenable (this was also proved independently by Phillips). Using this theorem we will prove that if $G$ is amenable, then $F^p(G,\sigma) \cong F^p_\lambda(G,\sigma)$. 
 	%To prove this we first need some result form Austad's paper \cite{are}.
 	
 	We say that a locally compact group $G$ is \emph{amenable} if for every $\varepsilon > 0$ and for every compact subset $K \subseteq G$ there exists a compact subset $ F\subseteq G$ such that $\mu(FK \Delta F) \leq \varepsilon\mu(F)$. 
 	
 	\begin{example}
 		\begin{enumerate}
 			\item  	Every compact group is amenable.
 			\item The extension of two amenable groups is also amenable. Given a locally compact group $G$ and $\sigma$ a continuous 2-cocycle on $G$, we have the following short exact sequence
 			$$1\rightarrow \mathbb{T} \rightarrow G_\sigma \rightarrow G\rightarrow 1\,.$$
 			Therefore, if $G$ is amenable, then the Mackey group $G_\sigma$ is also amenable. 
 		\end{enumerate}
 		
 	\end{example}

 	Let $p \in [1,\infty)$, let $G$ be a locally compact group and $\sigma$ a continuous 2-cocycle for $G$, and $E$ an $L^p$-space. A \emph{$\sigma$-projective isometric representation} of $G$ is a strongly continuous map $\pi \colon G \rightarrow \mathcal{U}(\mathcal{B}(E))$ such that
 	\begin{enumerate}
 		\item $\pi(x_1)\pi(x_2) = \sigma(x_1,x_2)\pi(x_1x_2)$ and
 		\item $\pi(e) = \textrm{Id}$.
 	\end{enumerate}
 	
 	It is well known that there is a natural bijective correspondence between $\sigma$-projective isometric representations of $G$ on $E$ and non degenerate $\sigma$-projective representations $L^1(G,\sigma) \rightarrow \mathcal{B}(E)$. If $\pi$ is a $\sigma$-projective isometric representation, then the induced  non degenerated representation $\pi \colon L^1(G,\sigma) \rightarrow \mathcal{B}(E)$ and $f\in L^{1}(G,\sigma)$ is given by 
 	$$
 	\pi(f)(\xi)(x) = \int_G f(y)\pi(y)\xi(x) dy
 	$$
 	for all $\xi\in E$ and $x\in G$, and it is called the \emph{integrated form of $\pi$}.
 	
 	\begin{definition}
 		Let $G$ be a locally compact group, let $\sigma$ be a continuous 2-cocycle in $G$ and let $p \in [1,\infty)$. We define $F^{p}(G,\sigma)$ as the \emph{full $\sigma$-twisted group $L^{p}$-operator algebra} as the completion of $L^{1}(G,\sigma)$ in the norm
 		$$
 		\norm{f}_{F^{p}} := \textrm{sup}\{\norm{\varphi(f)} \colon \varphi \colon L^{1}(G,\sigma) \rightarrow \mathcal{B}(E) \textrm{ is a contractive algebra homomorphism}\}\,,
 		$$
 		where $E$ ranges over all $L^{p}$-spaces. 
 	\end{definition}
 	
 	The argument of  \cite[Proposition 4.6]{gardella2019modern} shows that this is in fact an $L^p$-operator algebra.
 	
 	\begin{remark}
 		By the argument of \cite[Proposition 2.3]{Gardella_2018}, the full algebra can equivalently be defined as the completion of $L^1(G,\sigma)$ with respect to non-degenerate contractive representations. 
 	\end{remark}

 	A $\sigma$-projective isometric representation of $G$ induces an isometric representation of $G_\sigma$. This is done by sending $\pi \colon G \rightarrow \mathcal{U}(\mathcal{B}(L^p(G)))$ to $\pi_\sigma \colon G_\sigma \rightarrow \mathcal{U}(\mathcal{B}(L^p(G)))$ where 
 	$$
 	\pi_\sigma(\gamma,x) := \overline{\gamma}\pi(x)\,, 
 	$$
 	for every $(\gamma,x)\in G_\sigma$.
 	We can map  $L^p(G)$ isometrically to a subspace of $L^p(G_\sigma)$ with the map $j\colon L^p(G,\sigma) \rightarrow L^p(G_\sigma)$ given by
 	$$
 	j(f)(\gamma,x) := \gamma f(x)\,,
 	$$
 	for every $(\gamma,x)\in G_\sigma$.
 	By \cite[Lemma 3.3]{are}, the embedding is isometric and we thus have that $j(L^p(G))$ is a closed subspace of $L^p(G_\sigma)$ for $p \in [1,\infty]$. For $p=1$ the embedding is also an algebra homomorphism, and so we have that $j(L^1(G,\sigma))$ is a closed subalgebra of $L^1(G_\sigma)$. By expanding the functions of $L^p(G_\sigma)$ as Fourier series with respect to the second argument we can get an explicit description of the subspace $j(L^p(G,\sigma))$. For any $\xi\in L^p(G_\sigma)$ and any $x \in G$ we have that $\gamma \mapsto \xi(\gamma,x)$ is a function in $L^p(\mathbb{T}) \subseteq L^{p}(G_\sigma)$, which means that the Fourier coefficients 
 	$$
 	\xi_k(x) = \int_\mathbb{T}\xi(x,\gamma) \overline{\gamma}^nd\gamma
 	$$
 	are well defined, and that the resulting Fourier series
 	$$
 	\xi(\gamma,x) = \sum_{k\in\mathbb{Z}} \xi_k(x)\gamma^k
 	$$
 	converges in $L^p(\mathbb{T})$. By \cite[Lemma 3.4]{are} we have that
 	$$
 	j(L^p(G)) =\{\xi\in L^p(G_\sigma) \colon \xi_k = 0 \textrm{ for } k\neq 1 \}
 	$$
 	for $p\in[1,\infty]$.
 	The following proposition is \cite[Proposition 3.5]{are}.
 	
 	\begin{proposition}
 		\label{prev}
 		Let $G$ be a locally compact group and let $\sigma$ be a continuous 2-cocycle for $G$. Let $p \in [1,\infty)$, let $\xi \in L^1(G_\sigma)$ and let $\zeta\in L^p(G_\sigma)$. Then
 		$$
 		(\xi\ast \zeta)(\gamma,x) = \sum_{n\in \mathbb{Z}}(\xi_n\ast_{\sigma^n}\zeta_n)(x)\gamma^n\,.
 		$$
 	\end{proposition}
 	
 	Let $f \in L^1(G,\sigma)$. From \cite[Lemma 3.3]{are} it follows that
 	$$
 	\lambda_p(j(f))j(\xi) = j(f) \ast j(\xi) = j(f\ast_\sigma \xi)  = j(\lambda_p^\sigma(f)\xi)
 	$$
 	for all $\xi \in L^p(G)$, and from the Proposition \ref{prev} we deduce that $\norm{\lambda_p^\sigma(f)}_{\mathcal{B}(L^p(G))} = \norm{\lambda_p(j(f))}_{\mathcal{B}(L^p(G_\sigma))}$. It follows that $j$ maps $F_\lambda^p(G,\sigma)$ isometrically to a closed subalgebra of $F_\lambda^p(G_\sigma)$.
 	
 	Let $\pi$ be a $\sigma$-projective representation of $G$ on $E$ for some $L^p$-space $E$. Let $f \in L^1(G,\sigma)$, we have that
 	$$
 	\pi_\sigma(j(f))(\xi) = \int_G\int_\mathbb{T}\alpha f(y)\overline{\alpha}\pi(y)\xi d\alpha dy = \pi(f)\xi\,,
 	$$
 	for every $\xi\in E$, which implies that $\norm{\pi(f)}_{\mathcal{B}(E)} = \norm{\pi_\sigma(j(f))}_{\mathcal{B}(E)}$.
 	
 	\begin{theorem}\label{amenable_twisted}
 		Let $p \in (1,\infty)$, let $G$ be a locally compact group and let $\sigma$ be a continuous 2-cocycle for $G$. If $G$ is amenable, then there is an isometric isomorphism
 		$$
 		F^p(G,\sigma) \cong F_\lambda^p(G,\sigma).
 		$$
 	\end{theorem}
 	
 	\begin{proof}
 		Assume that $G$ is amenable, since $\mathbb{T}$ is compact, this implies that $G_\sigma$ is amenable, and by \cite[Theorem 3.7]{Gardella_2015} we have that $F_\lambda^p(G_\sigma) \cong F^p(G_\sigma)$. Let $f \in L^1(G,\sigma)$, and let $\pi$ be any $\sigma$-projective representation of $G$ on some $L^p$-space $E$. We want to show that $\|\pi(f)\|_{\mathcal{B}(E)} \leq \|\lambda_p^\sigma(f)\|_{\mathcal{B}(L^p(G))}$. Assume that this is not the case, i.e., that $\|\pi(f)\|_{\mathcal{B}(E)} > \|\lambda_p^\sigma(f)\|_{\mathcal{B}(L^p(G))}$. Note that $\pi$ induces a representation $\pi_\sigma$ of $G_\sigma$ on $E$ and that $\|\pi(f)\|_{\mathcal{B}(E)} = \|\pi_\sigma(j(f))\|_{\mathcal{B}(E)}$. We then have, by the definition, that $\|\pi_\sigma(j(f))\|_{\mathcal{B}(L^p(G_\sigma))} \leq \|j(f)\|_{F^p(G_\sigma)}$, but since $G$ is amenable  $\|j(f)\|_{F^p(G_\sigma)} = \|\lambda_p(j(f))\|_{\mathcal{B}(L^p(G_\sigma))}$. Thus, we have that
 		\begin{align*}
 			\|\pi(f)\|_{\mathcal{B}(E)} &> \|\lambda_p^\sigma(f)\|_{\mathcal{B}(L^p(G))} \\
 			&= \|\lambda_p(j(f))\|_{\mathcal{B}(L^p(G_\sigma))} \\ 
 			&= \|j(f)\|_{F^p(G_\sigma)} \\
 			&\geq \|\pi_\sigma(j(f))\|_{\mathcal{B}(E)} = \|\pi(f)\|_{\mathcal{B}(E)},
 		\end{align*}
 		which is a contradiction. It follows that $\|f\|_{F^p(G,\sigma)} = \|\lambda_p^\sigma(f)\|$ for all $f \in L^1(G,\sigma)$.
 	\end{proof}
 	\begin{remark}
 		For $p=1$, one can easily deduce, as  in \cite[Proposition 4.9]{gardella2019modern}, that $F^1(G,\sigma) = F^1_\lambda(G,\sigma) = L^1(G,\sigma)$.
 		
 	\end{remark}

 	As a consequence of this the reduced twisted group algebra generated by an amenable group can be characterized in terms of generators and relations.
 	
 	\begin{example}
 		Define the function $\sigma \colon \mathbb{Z}^{2} \times  \mathbb{Z}^{2} \rightarrow \mathbb{T}$ by $\sigma_\theta((m,n),(p,q)) = e^{2\pi inp\theta}$ for some irrational number $\theta \in \mathbb{R} \setminus{\mathbb{Q}}$. This is a continuous 2-cocycle on $\mathbb{Z}^2$ and for $p=2$, we have that $F^2(\mathbb{Z}^2,\sigma_\theta) = C^*(\mathbb{Z}^2,\sigma_\theta)$, which is the irrational rotation algebra, denoted by $A_\theta$ \cite{Rieffel}. Analogous to this, we will call $F^p(\mathbb{Z}^{2},\sigma_\theta)$ \emph{the $p$-irrational rotation} and denote it by $A^p_\theta$. Write $U_\theta = \delta_{(0,1)}$ and $V_\theta = \delta_{(0,1)}$, these are the generators of $F^p(\mathbb{Z}^{2},\sigma_\theta)$ and we have that $U_\theta V_\theta = e^{2\pi i\theta}V_\theta U_\theta$. For $p=2$ and $\theta = 0$ we have, by the Gelfand transform, that $F^2(\mathbb{Z}^2) \cong C(\mathbb{T}^2)$, the space of continuous function on the 2-torus. This is why this algebra is also known as the non commutative torus.
 		
 		Note that $\mathbb{Z}^{2}$ is an amenable group and as we have just shown this implies that $A^p_\theta =F^p(\mathbb{Z}^{2},\sigma_\theta) \cong F^p_\lambda(\mathbb{Z}^{2},\sigma_\theta)$. Furthermore, let $p\in (1,\infty)\setminus\{2\}$. Then, by Theorem \ref{main-group}, we have that $A^p_\theta$ and $A^p_\phi$ are isometrically isomorphic  if and only if $\sigma_\theta \sim \sigma_\phi$, that is, if there exists a continuous map $\gamma \colon \mathbb{Z}^{2} \rightarrow \mathbb{T}$ such that
 		
 		\begin{align*}
 			e^{2\pi inp\theta}e^{-2\pi inp\phi} = e^{2\pi inp(\theta-\phi)} = \gamma(m,n)\gamma(p,q)\overline{\gamma(m+p,n+q)}.
 		\end{align*}
 		
 		Note that $\gamma(1,0)\gamma(0,1)\overline{\gamma(1,1)} = 1$ and $\gamma(0,1)\gamma(1,0)\overline{\gamma(1,1)} = e^{2\pi i(\theta-\phi)}$, and so $e^{2\pi i(\theta-\phi)} = 1$ which means that $\theta-\phi \in \mathbb{Z}$. Conversely, if $\theta-\phi \in \mathbb{Z}$, then setting $ \gamma \equiv 1$ shows that 
 		$\sigma_\theta \sim \sigma_\phi$. 
 		
 		We would like to recall that when $p=2$ we have that $A_\theta\cong A_{1-\theta}$ \cite{Rieffel1}, so in this case the $C^*$-algebra  and the $L^p$-operator algebra non commutative torus behave differently.  This was observed before by Gardella and Thiel.
 	\end{example}

 	% \begin{proposition}
 		% Let $A$ be a unital $L^p$-operator algebra. Then $\mathcal{G}_A$ is a locally compact (not necessarily Hausdorff), effective, étale groupoid, and $\mathcal{G}^{(0)}$ is naturally
 		% homeomorphic to $X_A$.
 		% \end{proposition}
 	
 	% \begin{proof}
 		% Groupoid of germs are in general étale and effective (need referance!). localy compactness follows from continuity of source and range map, and the compactness of the unit space. 
 		% \end{proof}
 	
 	\section{The reduced twisted groupoid $L^p$-operator algebra}

 	In this section we will define the reduced twisted groupoid $L^p$-operator algebra along the lines of the work of \cite{choi2021rigidity,RenaultJean1980AGAt,2017arXiv171010897S}.

 	\subsection{Étale groupoids and twists.}
 	A groupoid $\mathcal{G}$  is a small category with inverses. We will denote by $\mathcal{G}^{(0)}$ the units of $\mathcal{G}$ and the source and the range of $\gamma\in \mathcal{G}$ by $s(\gamma)$ and $r(\gamma)$ respectively. Two elements $\alpha,\beta\in \mathcal{G}$ are composable  whenever $s(\alpha) = r(\beta)$, and we will denote by $\mathcal{G}^{(2)}$ the set of \emph{composable pairs}. Every unit can be identified with the identity morphism at the unit, so since every $\gamma\in \mathcal{G}$ has an inverse $\gamma^{-1}$, the unit space is under these identifications the subset of $\mathcal{G}$ of elements of the form $\gamma\gamma^{-1}$. The range map and source map $r,s \colon \mathcal{G} \rightarrow \mathcal{G}^{(0)}$ are maps given by 
 	\begin{align*}
 		r(\gamma)&= \gamma\gamma^{-1} &&\textrm{and} &s(\gamma)= \gamma^{-1}\gamma\,.
 	\end{align*}

 	For each unit $x \in \mathcal{G}^{(0)}$, we define
 	$$\mathcal{G}^{x} := \{\gamma \in \mathcal{G} \colon r(\gamma) = x \}\qquad\text{and}\qquad
 	\mathcal{G}_{x} := \{\gamma \in \mathcal{G} \colon s(\gamma) = x \}\,,$$
 	and $\mathcal{G}_{x}^{x}:= \mathcal{G}_{x} \cap \mathcal{G}^{x} = \{ \gamma \in \mathcal{G} \colon r(\gamma) = s(\gamma) = x\}$ is closed under product and inversion and therefore it is a group. We call the group  $\mathcal{G}_{x}^{x}$ the \emph{isotropy group} at $x$. One says that $x$ has \emph{trivial isotropy} if $\mathcal{G}_{x}^{x}$ only contains $x$. The set $\mathcal{G}^{\prime} = \{ \gamma \in \mathcal{G} \colon s(\gamma) = r(\gamma) \}$ is called the isotropy bundle. A \emph{topological groupoid} is a groupoid $\mathcal{G}$ endowed with a topology that makes the structure maps continuous. The topology on $\mathcal{G}^{(2)}$ is the relative topology inherited from the product topology in $\mathcal{G} \times \mathcal{G}$.

 	\begin{definition}
 		A locally compact groupoid $\mathcal{G}$ is called \emph{étale} if the range map $r \colon \mathcal{G} \rightarrow \mathcal{G}^{(0)}$ is a local homeomorphism.
 	\end{definition}
 	The trivial example of an étale groupoids are the discrete groups. 
 	\begin{definition}
 		Let $\mathcal{G}$ be a locally compact étale groupoid. $\mathcal{G}$ is said to be \emph{topological principal} if the set of points in $\mathcal{G}^{(0)}$ with trivial isotropy is dense in $\mathcal{G}^{(0)}$. $\mathcal{G}$ is said to be \emph{effective} if the interior of $\mathcal{G}^\prime$ is $\mathcal{G}^{(0)}$.
 	\end{definition}
 	
 	If $\mathcal{G}$ is a topologically principal and Hausdorff étale groupoid, then $\mathcal{G}$ is effective. The converse is not necessarily true. 
 	
 	A subset $B\subseteq \mathcal{G}$ is called a \emph{bisection} if there exists an open subset $U \subseteq \mathcal{G}$ containing $B$ such that the restriction of the range and source maps to $U$ are local homeomorphism. If $B$ is open then we call $B$ an \emph{open bisection} and the restrictions of the range and source maps to $B$ are local homeomorphisms.

 	\begin{definition}
 		A semigroup $S$ is called an \emph{inverse semigroup} if for each $s\in S$ there is a unique $t\in S$ such that 
 		$$
 		sts = s\qquad\text{and}\qquad tst = t\,.
 		$$
 		We write the element $t$ as $s^{\sharp}$. The map $s \mapsto s^{\sharp}$ is called the involution on $S$.
 	\end{definition}
 	
 	A \emph{partial homeomorphism} of a topological space $X$ is a homeomorphism $U \rightarrow V$ between open subsets $U$ and $V$ of $X$. Let $X$ be a compact Hausdorff space. We denote by $\textrm{Homeo}_{par}(X)$ the set of partial homeomorphisms of $X$, which forms an inverse semigroup. Let $\varphi \in \textrm{Homeo}_{par}(X)$. We denote $\textrm{dom}(\varphi)$ for the domain of $\varphi$ in $X$.

 	Let $X$ be a compact Hausdorff space, and let $S$ be an inverse subsemigroup of $\textrm{Homeo}_{par}(X)$. \emph{The groupoid of germs of $S$}, $\mathcal{G}(S)$, is defined as follows: On the set 
 	$$
 	\{(\varphi,x) \in S \times X \colon \varphi \in S, x \in \textrm{dom}(\varphi) \}
 	$$
 	we define the following equivalence class where $(\varphi,x) \sim (\phi,y) $ whenever $x=y$ and there exists a neighborhood $U$ of $x$ in $X$ such that ${\varphi}_{\restriction_U} = {\phi}_{\restriction_U}$. We write $[\varphi,x]$ for the equivalence class of $(\varphi,x)$. Then, $\mathcal{G}(S)$ has a natural groupoid structure with $r([\varphi,x]) = \varphi(x)$ and $s([\varphi,x]) = x$, multiplication given by
 	$$
 	[\varphi,\phi(y)][\phi,y] = [\varphi\circ \phi,y]
 	$$
 	and inverse given by
 	$$
 	[\varphi,x]^{-1} = [\varphi^{-1},\varphi(x)]\,.
 	$$
 	The groupoid of germs $\mathcal{G}(S)$ becomes an étale groupoid under the topology given by basic open sets $\mathcal{U}(U,\varphi) = \{[\varphi,x] \colon x \in U\}$ indexed by elements $ \varphi\in S$ and open subsets $U\subseteq \textrm{dom}(\varphi)$. The unit space of $\mathcal{G}(S)$ can be canonically identified with $X$ and it is thus a compact Hausdorff space. For details see \cite[Section 3]{renault2008cartan}. % Moreover, it was proved in \cite[Lemma 4.3]{Nyland} that $\mathcal{G}(S)$ is Hausdorff if and only if $\overline{\{ x\in X: \gamma(x)\neq x\}}$ is clopen for every $\gamma\in S$.
 	
 	\begin{remark}
 		Let $\mathcal{G}$ be an étale groupoid. Then the open bisections form an inverse semigroup denoted by $\mathcal{S}(\mathcal{G})$. Note that given two subsets $U,V \in \mathcal{G}$ composition is given by $UV = \{uv \colon s(u) = r(v)\}$. It follows that for two open bisections the composition is also an open bisection and that the composition is associative. Given an open bisection $B$ we also have $BB^{-1}B = B$ and $B^{-1}BB^{-1} = B^{-1}$.
 	\end{remark}
 	
 	\begin{remark}
 		\label{beta}
 		Note that any open bisection $B \in \mathcal{S}(\mathcal{G})$ defines a homeomorphism $\beta_B \colon s(B) \rightarrow r(B)$ given by $\beta_B(x) = r(Bx)$ for all $x \in s(B)$. Moreover, since the set of partial homeomorphisms is an inverse semigroup, the induced map $\beta \colon S(\mathcal{G}) \rightarrow \text{Homeo}_\textrm{par}(\mathcal{G}^{(0)})$ is an inverse semigroup homomorphism. We let $\mathcal{P}(\mathcal{G})$ denote the image of $\beta$. By  \cite[Corollary 3.4]{renault2008cartan}, the groupoid of germs of $\mathcal{P}(\mathcal{G})$ is isomorphic to $\mathcal{G}$ if and only if $\mathcal{G}$ is effective. Moreover, if this is the case, then $\beta$ identifies $\mathcal{S}(\mathcal{G})$ bijectively with $\mathcal{P}(\mathcal{G})$.
 	\end{remark}
 	
 	We will follow \cite[Chapter 5]{2017arXiv171010897S} in defining the twist.
 	
 	\begin{definition}
 		\label{twist_def}
 		Let $\mathcal{G}$ be an étale groupoid. A \emph{twist} over $\mathcal{G}$ is a sequence 
 		$$\mathcal{G}^{(0)} \times \mathbb{T} \xrightarrow{i} \mathcal{E} \xrightarrow{\pi} \mathcal{G}$$
 		where $\mathcal{G}^{(0)} \times \mathbb{T}$ is regarded as a trivial group bundle with fibres $\mathbb{T}$, $\mathcal{E}$ is a locally compact Hausdorff groupoid, and $i$ and $\pi$ are continuous groupoid homomorphisms that restrict to homeomorphisms of unit spaces, and such that
 		\begin{enumerate}
 			\item $i$ is injective,
 			\item $\mathcal{E}$ is locally a trivial $\mathcal{G}$-bundle, i.e., every point $\gamma \in \mathcal{G}$ has a bisection neighborhood $U$ such that there exists a continuous section $S \colon U \rightarrow \mathcal{E}$ satisfying $\pi \circ S = \text{id}_U$ and such that the map $(\gamma,z) \mapsto i(r(\gamma),z)S(\gamma)$ is a homeomorphism $U \times \mathbb{T} \rightarrow \pi^{-1}(U)$,
 			\item $i(\mathcal{G}^{(0)} \times \mathbb{T})$ is central in $\mathcal{E}$, i.e., $i(r(\varepsilon),z)\varepsilon = \varepsilon i(s(\varepsilon),z)$ for all $\varepsilon \in \mathcal{E}$ and $z \in \mathbb{T}$; and
 			\item $\pi^{-1}(\mathcal{G}^{(0)}) = i(\mathcal{G}^{(0)} \times \mathbb{T})$.
 		\end{enumerate}
 	\end{definition}
 	
 	If $\mathcal{E}$ is a twist over $\mathcal{G}$ for $z \in \mathbb{T}$ and $\varepsilon \in \mathcal{E}$, we write  $z \cdot \varepsilon  = i(r(\varepsilon),z)\varepsilon$ and $\varepsilon \cdot z = \varepsilon i(s(\varepsilon),z)$. Note that $\varepsilon \cdot z =  z \cdot \varepsilon$ since $\mathcal{E}$ is central in $\mathcal{G}$. We identify $\mathcal{E}^{(0)}$ with $\mathcal{G}^{(0)}$ via the map $x \mapsto i(x,1)$. 
 	
 	\begin{lemma}
 		Let $\mathcal{G}$ be an étale groupoid and let $\mathcal{E}$ be a twist over $\mathcal{G}$. If two elements $\varepsilon,\delta \in \mathcal{E}$ satisfy $\pi(\varepsilon) = \pi(\delta)$, then there exists $z \in \mathbb{T}$ such that $z \cdot \varepsilon = \delta$.
 	\end{lemma}
 	
 	\begin{notation}
 		Let $\delta,\varepsilon,\gamma \in \mathcal{E}$ with $\pi(\delta) = \pi(\varepsilon) = \pi(\gamma)$. We will use the notation $\delta = z(\delta,\varepsilon) \cdot \varepsilon$, where $z(\delta,\varepsilon)$ is the element $z \in \mathbb{T}$ such that $\delta = z \cdot \varepsilon$. Note that $\overline{z(\delta,\varepsilon)}= z(\varepsilon,\delta)$, and $z(\varepsilon,\delta)z(\delta,\gamma) = z(\varepsilon,\gamma)$.
 	\end{notation}
 	
 	\begin{definition}
 		\label{nor_2-cocyle}
 		A \emph{normalized continuous 2-cocycle} for a topological groupoid is a continuous map $\sigma \colon \mathcal{G}^{(2)} \rightarrow \mathbb{T}$ satisfying the following:
 		\begin{enumerate}
 			\item $\sigma(r(\gamma),\gamma) = \sigma(\gamma, s(\gamma))$ = 1 for all $\gamma \in \mathcal{G}$.
 			\item $\sigma(\alpha,\beta)\sigma(\alpha\beta,\gamma) = \sigma(\beta,\gamma)\sigma(\alpha,\beta\gamma)$ whenever $(\alpha,\beta),(\beta,\gamma) \in \mathcal{G}^{(2)}$.
 		\end{enumerate}
 	\end{definition}
 	
 	\begin{definition}
 		Let $\sigma, \omega$ be two normalized continuous 2-cocycles for $\mathcal{G}$, we say that $\sigma$ is \emph{cohomologous} to $\omega$, if there exists a continuous function $\gamma \colon \mathcal{G} \rightarrow \mathbb{T}$ such that $\sigma(\alpha,\beta)\overline{\omega(\alpha,\beta)} = \gamma(\alpha)\gamma(\beta)\overline{
 			\gamma(\alpha\beta)}$ for all $(\alpha,\beta) \in \mathcal{G}^{(2)}$. \end{definition}

 	\begin{example}
 		Let $\mathcal{G}$ be a locally compact étale Hausdorff groupoid. If $\sigma$ is a continuous normalized 2-cocycle on $\mathcal{G}$, then we can make $\mathcal{G} \times \mathbb{T}$ into a groupoid $\mathcal{E}_\sigma$. The unit space, range and source maps are given as usual, but multiplication is given by
 		$(\alpha,z)(\beta,w) = (\alpha\beta,wz\sigma(\alpha,\beta))$ and inversion is given by
 		$(\alpha,z)^{-1} = (\alpha,\overline{\sigma(\alpha^{-1},\alpha)z})$. The groupoid $\mathcal{E}_\sigma$ is analogous to the Mackey group $G_\sigma$ in Section 3. The set inclusion $i \colon \mathcal{G}^{(0)} \times \mathbb{T} \rightarrow \mathcal{E}_\sigma$ and the projection $\pi \colon \mathcal{E}_\sigma \rightarrow \mathcal{G}$ given by $\pi(\gamma,z) = \gamma$ are groupoid homomorphisms. One can then show that $\mathcal{E}_\sigma$ is a twist over $\mathcal{G}$ with respect to $i$ and $\pi$.
 		
 		We can also recover the cohomology class of $\sigma$ from the twist $\mathcal{E}_\sigma \rightarrow \mathcal{G}$. Let $S$ be any continuous section for $\sigma$, i.e, a continuous map $S \colon \mathcal{G} \rightarrow \mathcal{E}_\sigma$ such that $\pi \circ S = \textrm{id}_\mathcal{G}$. For $(\alpha,\beta) \in \mathcal{G}^{(2)}$ we have $\pi\left(S(\alpha)S(\beta)S(\alpha\beta)^{-1}\right) = r(\alpha)$, and so there is a unique element (dependent on $\alpha$ and $\beta$) $\omega(\alpha,\beta) \in \mathbb{T}$ such that $S(\alpha)S(\beta)S(\alpha\beta)^{-1} = (r(\alpha), \omega(\alpha,\beta))$. The resulting map $\omega \colon \mathcal{G}^{(2)} \rightarrow \mathbb{T}$ is a continuous 2-cocycle. Let $S^\prime$ be another continuous section for $\sigma$, and let $\omega^\prime$ be defined as $\omega$, but with respect to $S^\prime$. Let $b \colon \mathcal{G} \rightarrow \mathbb{T}$ be the map satisfying the equality $S(\alpha)^{-1}S^\prime(\alpha) = (r(\alpha),b(\alpha))$ for all $\alpha \in \mathcal{G}$, this is in fact a 1-cochain. Then $\omega^{-1}\omega^{\prime} \colon \mathcal{G}^{(2)} \rightarrow \mathbb{T}$ given by $(\omega^{-1}\omega^{\prime})(\alpha,\beta) = \omega(\alpha,\beta)^{-1}\omega^{\prime}(\alpha,\beta)$ is equal to the 2-coboundary obtained from the 1-cochain $b$. Thus, the cocycle obtained from different choices of continuous sections $S$ are cohomologous. If we let $S$ be the continuous section given by $S(\gamma) = (\gamma,1)$ for all $\gamma \in \mathcal{G}$, then $\omega = \sigma$. Thus, the cohomology class of $\sigma$ is equal to the cohomology class obtained from any continuous section  $\mathcal{G} \rightarrow \mathcal{E}_\sigma$.

 	\end{example}
 	
 	In general if $\mathcal{E}$ is a twist over $\mathcal{G}$ that admits a continuous section $S \colon \mathcal{G} \rightarrow \mathcal{E}$, then there is a continuous 2-cocycle defined by $S(\alpha)S(\beta)S(\alpha\beta)^{-1} = i\left(s(\alpha),\sigma(\alpha,\beta)\right)$ for all $(\alpha,\beta) \in \mathcal{G}^{(2)}$, and thus an isomorphism $\mathcal{E} \cong \mathcal{E}_\sigma$ that is equivariant for $i$ and $\pi$. In that case, $\mathcal{E}$ is isomorphic to a twist coming from a continuous 2-cocycle. 
 	%However, not every twist admits a continuous section, so the notion of a twist is formally stronger than the one of continuous 2-cocycle.
 	
 	\subsection{The twisted groupoid $L^p$-operator algebra.}
 	
 	\begin{definition}
 		Let $\mathcal{E}$ be a twist over a locally compact étale groupoid $\mathcal{G}$. We define
 		$$
 		\Sigma_c(\mathcal{G}; \mathcal{E}) = \{f \in C_c(\mathcal{E}) \colon f(z \cdot \varepsilon) = zf(\varepsilon)\, \text{ for all }
 		\varepsilon \in \mathcal{E}\,, z \in \mathbb{T}\}\,.
 		$$
 	\end{definition}
 	
 	Fix $\gamma \in \mathcal{G}$. For any element $\delta \in \pi^{-1}(\gamma)$ we have a homeomorphism $\mathbb{T} \cong \pi^{-1}(\gamma)$ given by $z \mapsto z \cdot \delta$. We define a measure on $\pi^{-1}(\gamma)$ by pulling back the Haar measure on $\mathbb{T}$. This is independent of the choice of $\delta \in \pi^{-1}(\gamma)$ since the Haar measure on $\mathbb{T}$ is rotation invariant. For every $x \in \mathcal{G}^{(0)}$ we endow $\mathcal{E}_x$ ($\mathcal{E}^x$) with the measure $\nu_x$ ($\nu^x$) that agrees with the pulled backed copy of the Haar measure on $\pi^{-1}(\gamma)$ for each $\gamma \in \mathcal{G}_x$ ($\mathcal{G}^x$). Note that $\pi^{-1}(\gamma)$ has measure 1. We have that $\Sigma_c(\mathcal{G}; \mathcal{E})$ has a $*$-algebra structure \cite[Lemma 5.1.9]{2017arXiv171010897S}.
 	
 	\begin{lemma}
 		Let  $\mathcal{E}$ be a twist over a locally compact étale groupoid $\mathcal{G}$. Then, $\Sigma_c(\mathcal{G}; \mathcal{E})$ forms a complex $\ast$-algebra with product given by convolution
 		$$
 		(f \ast g)(\varepsilon) = \int_{\mathcal{E}^{r(\varepsilon)}} f(\gamma)g(\gamma^{-1}\varepsilon) d\nu^{r(\varepsilon)} = \int_{\mathcal{E}_{s(\varepsilon)}} f(\varepsilon\gamma^{-1})g(\gamma) d\nu_{s(\varepsilon)}\,,
 		$$
 		for every $f,g\in \Sigma_c(\mathcal{G}; \mathcal{E})$,
 		and involution given by
 		$$
 		h^{\ast}(\varepsilon) = \overline{h(\varepsilon^{-1})}\,,
 		$$
 		for every $h\in \Sigma_c(\mathcal{G}; \mathcal{E})$.
 		Let $\varepsilon \in \mathcal{E}$, then for any choice of a (not necessarily continuous) section $S\colon\mathcal{G}\to \mathcal{E}$ of $\pi$, we have that
 		\begin{equation}
 			\label{section_conv}    
 			(f \ast g)(\varepsilon) = \sum_{\alpha \in \mathcal{G}^{r(\varepsilon)}}f(S(\alpha))g(S({\alpha})^{-1}\varepsilon) = \sum_{\alpha \in \mathcal{G}^{s(\varepsilon)}}f(\varepsilon S({\alpha})^{-1})g(S(\alpha))\,,
 		\end{equation}
 		for $f,g \in \Sigma_c(\mathcal{G};\mathcal{E})$.
 		There is an isomorphism 
 		\begin{equation}
 			\label{core-id}
 			C_c(\mathcal{G}^{(0)}) \cong \{f \in \Sigma_c(\mathcal{G}, \mathcal{E}) \colon \text{supp}(f) \subseteq i(\mathcal{G}^{0} \times \mathbb{T})\}
 		\end{equation}
 		that sends $f \in C_c(\mathcal{G}^{(0)})$ to $\Tilde{f}$, where $\Tilde{f}(i(x,z))= zf(x)$.
 	\end{lemma}

 	\begin{remark}
 		Each twist $\mathcal{E}$ over $\mathcal{G}$ determines a complex line bundle $\Tilde{\mathcal{E}}$ over $\mathcal{G}$, where $\Tilde{\mathcal{E}} = \mathbb{C} \times \mathcal{E} / \sim$ with $(z\overline{t},\gamma) \sim (z,t\cdot \gamma)$. If we denote the corresponding  equivalence class by $[z,\gamma]$, then $\Tilde{\mathcal{E}}$ is a line bundle over $\mathcal{G}$ with respect to the fiber map $p\colon \Tilde{\mathcal{E}} \rightarrow \mathcal{G}$ given by $p([z,\gamma]) \mapsto \pi(\gamma)$. We can regard elements of $\Sigma_c(\mathcal{G}; \mathcal{E})$ as sections of $\Tilde{\mathcal{E}}$, where the corresponding section of $f \in \Sigma_c(\mathcal{G}; \mathcal{E})$ is given by 
 		$$
 		\gamma \mapsto [f(\Tilde{\gamma}),\Tilde{\gamma}]
 		$$
 		for any choice of  $\Tilde{\gamma} \in \pi^{-1}(\gamma)$. This is well defined due to the equivalence relation.
 	\end{remark}
 	
 	Let $p\in[1,\infty)$. We denote the $p$-integrable $\mathbb{T}$-equivariant functions on $\mathcal{E}_x$ as $L^{p}(\mathcal{G}_x; \mathcal{E}_x)$. For $p=\infty$, we denote the $\mathbb{T}$-equivariant supremum bounded functions by $L^\infty(\mathcal{G}_x; \mathcal{E}_x)$.
 	
 	\begin{lemma}
 		\label{lp}
 		Let $p\in[1,\infty]$, let $\mathcal{E}$ be a twist over a locally compact étale groupoid $\mathcal{G}$, and let $x \in \mathcal{G}^{(0)}$. Then $L^{p}(\mathcal{G}_x; \mathcal{E}_x)$ and  $l^{p}(\mathcal{G}_x)$ are isometrically isomorphic.
 	\end{lemma}
 	
 	\begin{proof}
 		Let $S \colon \mathcal{G} \rightarrow \mathcal{E}$ be a section of $\pi$. Then for every $\alpha \in \mathcal{E}_x$ there exists $z \in \mathbb{T}$ such that $\alpha = S({\pi(\alpha)})i(x,z)$.
 		Let $f \in l^p(\mathcal{G}_x)$. Then define the function $\widetilde{f} \in L^{p}(\mathcal{G}_x; \mathcal{E}_x)$ as $\widetilde{f}(\alpha):= zf(\pi(\alpha))$, where $z \in \mathbb{T}$ is given by $\alpha = S(\pi(\alpha))i(x,z)$ (also denoted $z(\alpha,S(\pi(\alpha)))$. For $p\in[1,\infty)$  we have that
 		$$
 		\int_{\mathcal{E}_x} |\Tilde{f}(\alpha)| d\nu_x = \sum_{\gamma \in \mathcal{G}_x} \int_{\mathbb{T}} |\Tilde{f}(z\cdot S(\gamma))|dz = \sum_{\gamma \in \mathcal{G}_x}|f(\gamma)| \int_{\mathbb{T}} |z| dz = \sum_{\gamma \in \mathcal{G}_x}|f(\gamma)|
 		$$
 		and for $p=\infty$ we have that
 		$$
 		\sup_{\alpha \in \mathcal{E}_x} |\Tilde{f}(\alpha)| = \sup_{\gamma \in \mathcal{G}_x}\sup_{\alpha \in \pi^{-1}(\gamma)} |zf(\gamma)| = \sup_{\gamma \in \mathcal{G}_x} |f(\gamma)|\,.
 		$$
 		It follows that the map is isometric for all $p\in[1,\infty]$.
 		The map $f \mapsto \widetilde{f}$ is clearly injective. To show surjectivity, let $g \in L^{p}(\mathcal{G}_x; \mathcal{E}_x)$. Given $\alpha\in \mathcal{E}_x$ there exists  $z \in \mathbb{T}$ such that $\alpha = z \cdot S(\pi(\alpha))$. Then  we have that $g(\alpha) = zg(S(\pi(\alpha)))$ and so the function $f\in l^p(\mathcal{G}_x)$ given by $\gamma \mapsto g(S(\gamma))$ for all $\gamma \in \mathcal{G}_x$ is such that $g(\alpha) = \Tilde{f}(\alpha)$ for all $\alpha \in \mathcal{E}_x$. Thus, the map is an isomorphism and it follows that $L^{p}(\mathcal{G}_x; \mathcal{E}_x) \cong l^{p}(\mathcal{G}_x)$.
 	\end{proof}
 	
 	Let $\mathcal{G}$ be an étale groupoid and $\mathcal{E}$ be a twist over $\mathcal{G}$. Given $x\in \mathcal{G}^{(0)}$, we define the left regular representation $\lambda_x \colon \Sigma_c(\mathcal{G}; \mathcal{E}) \rightarrow \mathcal{B}(L^{p}(\mathcal{G}_x; \mathcal{E}_x))$ by extension of the convolution formula, that is
 	$$
 	\lambda_x(f)\xi(\varepsilon) = \int_{\mathcal{E}_{x}} f(\varepsilon\gamma^{-1})\xi(\gamma) d\nu_{x} \,,
 	$$
 	for every $f\in \Sigma_c(\mathcal{G}; \mathcal{E})$, $\xi\in L^{p}(\mathcal{G}_x; \mathcal{E}_x)$ and $\varepsilon \in \mathcal{E}_x$. We have that 
 	$$
 	\norm{\lambda_x(f)} \leq \norm{f}_I = \max \{ \sup_{x}\int_{\mathcal{E}_x}|f(\alpha)|d\nu_x, \sup_{x}\int_{\mathcal{E}_x}|f^\ast(\alpha)|d\nu_x \}
 	\,,$$
 	so $\lambda_x$ is bounded.
 	To show this estimate, choose any section $S\colon\mathcal{E}_x\to \mathcal{G}_x$ of $\pi$. We then have that
 	\begin{align*}
 		\lambda_x(f)\xi(\varepsilon) = \int_{\mathcal{E}_{x}} f(\varepsilon\gamma^{-1})\xi(\gamma) d\nu_{x}  = \sum_{\alpha \in \mathcal{G}_{x}}f(\varepsilon S(\alpha^{-1}))\xi(S(\alpha))
 	\end{align*}
 	
 	and
 	\begin{align*}
 		\int_{\mathcal{E}_x} |f(\alpha)| d\nu_x = \sum_{\gamma \in \mathcal{G}_x} \int_{\mathbb{T}} |f(z\cdot S(\gamma))| dz = \sum_{\gamma\in \mathcal{G}_x} |f(S(\gamma))|\,,
 	\end{align*}
 	and since $|f(\gamma_1)| = |f(\gamma_2)|$ for all $\gamma_1,\gamma_2 \in \mathcal{E}_x$ with $\pi(\gamma_1) = \pi(\gamma_2)$,  we have by Lemma \ref{lp} that 
 	
 	$$
 	\sup_{\gamma \in \mathcal{E}_x} |f(\gamma)| = \sup_{\alpha \in \mathcal{G}_x} |f(S(\alpha))|\,.
 	$$
 	The estimate $\norm{\lambda_x(f)} \leq \norm{f}_I$ then follows as in the proof of  \cite[Proposition 4.2]{choi2021rigidity}.
 	
 	% \begin{remark}
 		% Let $\sigma$ be a 2-cocycle, and let $\mathcal{E}$ be twist constructed from $\sigma$ as i remark \ref{twist_cocycle}. If we let $S$ be the continuous section given by $S(\gamma) = (\gamma,1)$ for all $\gamma \in \mathcal{G}_x$. Then there is a isomorphism $L^{p}(\mathcal{G}_x, \mathcal{E}_x) \cong l^{p}(\mathcal{G}_x)$, that send $f \in l^p(\mathcal{G}_x)$ to the function $(\gamma,z) \mapsto zf(\gamma)$.
 		% \end{remark}

 	\begin{definition}
 		Let $p\in[1,\infty)$, and let $\mathcal{E}$ be a twist over a locally compact étale groupoid $\mathcal{G}$. We define \emph{the reduced twisted $L^{p}$-operator algebra}, denoted $F^p_{\lambda}(\mathcal{G};\mathcal{E})$, as the completion of $\Sigma_c(\mathcal{G}; \mathcal{E})$ in the norm
 		$$
 		\norm{f}_\lambda = \text{sup}_{x\in \mathcal{G}^{(0)}}\norm{\lambda_x(f)}\,.
 		$$
 	\end{definition}

 	That $F^p_{\lambda}(\mathcal{G};\mathcal{E})$ is an $L^p$-operator algebra follows since $L^p(\mathcal{G}_x;\mathcal{E}_x)$ is an $L^{p}$-space, and then $\lambda:=\bigoplus_{x\in \mathcal{G}^{(0)}} \lambda_x$ is an isometric representation of $F^p_{\lambda}(\mathcal{G};\mathcal{E})$ on the $L^p$-space $\bigoplus_{x\in \mathcal{G}^{(0)}} L^p(\mathcal{G}_x;\mathcal{E}_x)$. Note that if $\mathcal{G}^{(0)}$ is compact, then  $F^p_{\lambda}(\mathcal{G};\mathcal{E})$ is unital with unit $\lambda(1_{\mathcal{G}^{(0)}})$.
 	
 	\begin{example}
 		Let $\mathcal{G}$ be a locally compact étale groupoid. For the trivial twist $\mathcal{E} = \mathcal{G} \times \mathbb{T}$, the continuous map $\alpha \mapsto (\alpha,1)$ is a section for $\pi \colon \mathcal{G}\times \mathbb{T} \rightarrow \mathcal{G}$. The cocycle obtained from this section is the trivial cocyle, and it follows from (\ref{section_conv}) that $F^p_{\lambda}(\mathcal{G};\mathcal{G} \times \mathbb{T}) \cong F^p_{\lambda}(\mathcal{G})$, where $F^p_{\lambda}(\mathcal{G})$ is the reduced groupoid $L^p$-operator algebra defined in \cite{choi2021rigidity}. Now, let $\sigma$ be a continuous 2-cocycle for $\mathcal{G}$ and let $\mathcal{E}_\sigma$ be the twist over $\mathcal{G}$ constructed from $\sigma$.  We can then define the  $\sigma$-twisted convolution on $C_c(\mathcal{G})$ given by
 		$$
 		(f \ast_\sigma g)(\gamma) = \sum_{\alpha\in \mathcal{G}_s(\gamma)} f(\gamma\alpha^{-1})g(\alpha)\sigma(\gamma\alpha^{-1},\alpha) = \sum_{\alpha \in \mathcal{G}^{r(\gamma)}} f(\alpha)g(\alpha^{-1}\gamma)\sigma(\alpha,\alpha^{-1}\gamma)
 		$$
 		for $f,g \in C_c(\mathcal{G})$ and all $\gamma \in \mathcal{G}$. We denote the complex algebra formed by $\sigma$-twisted convolution by $C_c(\mathcal{G},\sigma)$. We can also define an involution product on $C_c(\mathcal{G},\sigma)$ given by
 		$$
 		f^{\ast_{\sigma}}(\gamma) = \overline{\sigma(\gamma^{-1},\gamma)f(\gamma^{-1})}\,.
 		$$
 		
 		For $p\in [1,\infty]$,
 		we define the mapping $\lambda_x^{\sigma} \colon C_c(G,\sigma) \rightarrow l^{p}(\mathcal{G}_x)$ to be the unique map such that $\lambda_x^{\sigma}(f) = f \ast_\sigma \xi$ for all $\xi \in C_c(\mathcal{G}_x)$. Then, we define the \emph{reduced $\sigma$-twisted norm} on $C_c(\mathcal{G},\sigma)$ by  
 		$$
 		\norm{f}_{\lambda^{\sigma}} = \sup \{\norm{\lambda_{x}^{\sigma}(f)}_{p} \colon x \in \mathcal{G}^{0} \}\,.
 		$$
 		We denote the completion of $C_c(\mathcal{G},\sigma)$ with respect to $\norm{\cdot}_{\lambda^{\sigma}}$ by $F^p_{\lambda}(\mathcal{G}, \sigma)$.

 		We have that 
 		$F^p_{\lambda}(\mathcal{G};\mathcal{E}_{\sigma})$ and $F^p_{\lambda}(\mathcal{G}, \sigma)$ are isometrically isomorphic. There is an algebra  isomorphism $\Sigma_c(\mathcal{G};\mathcal{E}_{\sigma}) \rightarrow C_c(\mathcal{G}, \sigma)$ that sends $f \mapsto \Tilde{f}$, where $\Tilde{f}(z,\gamma) = zf(\gamma)$ for every $(z,\gamma)\in \mathcal{E}_{\sigma}$. One can also show that $\norm{f}_{\lambda^\sigma} = \| \Tilde{f}\|_\lambda$, which means that the isomorphism extends isometrically to the closures.
 	\end{example}

 	\begin{remark}
 		Let $\beta,\varepsilon \in \mathcal{E}$, we define the point mass function on $L^{p}(\mathcal{G}_x;\mathcal{E}_x)$ as follows: $\delta_\beta(\varepsilon)  = z$ with $\pi(\beta) = \pi(\varepsilon)$ and $\varepsilon =z\cdot \beta$, otherwise $\delta_\beta(\varepsilon) = 0$.
 	\end{remark}
 	\begin{lemma}
 		\label{norm}
 		Let $p\in[1,\infty)$, and let $\mathcal{E}$ be a twist over a locally compact étale groupoid $\mathcal{G}$. Let $f \in \Sigma_c(\mathcal{G}; \mathcal{E})$. Then
 		$$
 		\norm{f}_{\infty} \leq \norm{f}_\lambda \leq \norm{f}_I.
 		$$
 		
 		Furthermore, if $B$ is a bisection, then for every $f$ with $\text{supp}(f) \subseteq \pi^{-1}(B)$ we have that $\norm{f}_\infty = \norm{f}_\lambda$. 
 	\end{lemma}
 	\begin{proof}
 		Let $f \in \Sigma_c(\mathcal{G}; \mathcal{E})$. Since $\norm{\lambda_x(f)} \leq \norm{f}_I$, the second inequality follows immediately. Let $\varepsilon \in \mathcal{E}$. For first inequality, we need to show that $|f(\varepsilon)| \leq \norm{f}_\lambda$. Set $x = s(\varepsilon)$, then 
 		\begin{align*}
 			\norm{f}_\lambda \geq \norm{\lambda_x(f)} \geq \norm{\lambda_x(f)\delta_x} = \|\sum_{\gamma \in \mathcal{G}_x} f(\gamma)\delta_\gamma\|_p \geq |f(\varepsilon)|\,.
 		\end{align*}
 		If $\textrm{supp}(f) \subseteq \pi^{-1}(B)$ for a bisection $B$, then a quick manipulation of terms shows that $\norm{f}_I = \norm{f}_\infty$, and it follows that $\norm{f}_\infty = \norm{f}_\lambda$. 
 	\end{proof}
 	
 	\begin{remark}
 		\label{pm}
 		Note that $\mathcal{G}^{(0)}$ is a bisection, and so for all $f \in C_c(G^{(0)})$ the equality $\norm{f}_I = \norm{f}_\infty$ holds.
 	\end{remark}
 	
 	For the rest of the section fix $p\in[1,\infty)$, a twist $\mathcal{E}$ over an étale groupoid $\mathcal{G}$ and a section $S:\mathcal{G}\to \mathcal{E}$ of $\pi$.

 	The identity map $\Sigma_c(\mathcal{G}; \mathcal{E}) \rightarrow C_0(\mathcal{E})$ extends to a linear contractive map $j \colon F^p_{\lambda}(\mathcal{G};\mathcal{E}) \rightarrow C_0(\mathcal{E})$ by the previous lemma. Given $a \in F^{p}_{\lambda}(\mathcal{G};\mathcal{E}_x)$ we will write $j_a$ for $j(a) \in C_0(\mathcal{E})$.
 	
 	Now let $q$ be the dual exponent of $p$, i.e,  $\frac{1}{p} + \frac{1}{q} = 1$ when $p\neq 1$ and $q=\infty$ when $p=1$. For $x \in \mathcal{G}^{(0)}$, using Proposition \ref{lp} we can identify the dual of $L^{p}(\mathcal{G}_{x};\mathcal{E}_x)$ with $L^{q}(\mathcal{G}_{x};\mathcal{E}_x)$, with the dual pairing defined as
 	$$\langle \xi\,,\eta \rangle = \sum_{\gamma\in \mathcal{G}_x}\xi(S(\gamma))\overline{\eta(S(\gamma))}\,,$$
 	for $\xi \in L^{p}(\mathcal{G}_{x};\mathcal{E}_x)$ and $\eta \in L^{q}(\mathcal{G}_{x};\mathcal{E}_x)$.  
 	
 	%Using the identification of Proposition \ref{lp} we define a operator $\Tilde{\lambda}_x \colon l^p(\mathcal{G}_x) \rightarrow l^p(\mathcal{G}_x)$ as follows: 
 	%Let $\xi \in l^p(\mathcal{G}_x)$, and let $\Tilde{\xi}$ be its associated element in  $L^p(\mathcal{E}_x;\mathcal{G}_x)$ (see %Proposition \ref{lp}), then we define  
 	%$$
 	%\lambda_x(f)(\Tilde{\xi})(\varepsilon) := \sum_{\alpha \in \mathcal{G}_x}f(\varepsilon\Tilde{\alpha}^{-1})\xi(\alpha)
 	%$$
 	%and so $\Tilde{\lambda}_x$ is given by
 	%$$
 	%\Tilde{\lambda}_x(f)(\Tilde{\xi}(\gamma)) = \sum_{\alpha \in \mathcal{G}_x}f(\Tilde{\gamma}\Tilde{\alpha}^{-1})\xi(\alpha)
 	%$$
 	%The reason we define this operator is that we need the properties of dual action is the following results. 
 	
 	\begin{proposition}
 		\label{map}
 		The map $j \colon F^{p}_{\lambda}(\mathcal{G};\mathcal{E}) \rightarrow C_0(\mathcal{E})$ is injective and we have
 		$$
 		j_a(\alpha) = \langle \lambda_{s(\alpha)}(a)(\delta_{s(\alpha)})\,,\delta_\alpha\rangle \,,
 		$$
 		for all $a \in F^{p}_{\lambda}(\mathcal{G};\mathcal{E})$ and all $\alpha \in \mathcal{E}$.
 	\end{proposition}
 	
 	%Note that the operator the point mass functions here are the functions as given in Remark \ref{pm} which we will here see as functions in $l^p(G_x)$ using the identify in the proof of Proposition \ref{lp}. For instance we write $\overline{\delta_\gamma}$ where $\gamma \in \mathcal{E}$ for the function $\alpha \rightarrow \overline{\delta_\gamma(\Tilde{\alpha})}$ for all $\alpha \in \mathcal{G}$, which is a function in $l^q(G_x)$. This is important as the function $\varepsilon \mapsto \overline{\delta_\gamma}(\varepsilon)$ for $\varepsilon \in \mathcal{E}$ is not $\mathbb{T}$-equivariant, and so is not a function in $L^p(\mathcal{G}_x,\mathcal{E}_x)$. 
 	\begin{proof}
 		For $f \in \Sigma_c(\mathcal{G}; \mathcal{E})$ and $\alpha\in \mathcal{E}$ we have that 
 		\begin{align*}
 			j_f(\alpha) & = \langle \lambda_{s(\alpha)}(f)(\delta_{s(\alpha)})\,,\delta_\alpha\rangle = \sum_{\gamma\in \mathcal{G}_x}\lambda_{s(\alpha)}(f)(\delta_{s(\alpha)})(S(\gamma))\overline{\delta_\alpha(S(\gamma))}\\
 			& =\lambda_{s(\alpha)}(f)(\delta_{s(\alpha)})(S(\pi(\alpha)))\overline{\delta_\alpha(S(\pi(\alpha)))}\\
 			& =\left( \int_{\mathcal{E}_{s(\alpha)}}f(S(\pi(\alpha))\eta^{-1})\delta_{s(\alpha)}(\eta)d\nu_{s(\alpha)} \right) \overline{\delta_\alpha(S(\pi(\alpha)))} \\
 			& =\left( \int_{\mathbb{T}}f(S(\pi(\alpha))(\overline{z}\cdot s(\alpha)))\delta_{s(\alpha)}(z\cdot s(\alpha))d z \right) \overline{\delta_\alpha(S(\pi(\alpha)))}\\
 			& =\left( \int_{\mathbb{T}}f(S(\pi(\alpha)))d z \right) \overline{\delta_\alpha(S(\pi(\alpha)))} =f(S(\pi(\alpha)))\overline{\delta_\alpha(S(\pi(\alpha)))} \\
 			& =f(S(\pi(\alpha))) \overline{z(S(\pi(\alpha)),\alpha)}= f(S(\pi(\alpha))) z(\alpha, S(\pi(\alpha)))\\
 			& =f(z(\alpha,S(\pi(\alpha)))\cdot S(\pi(\alpha)))=f(\alpha)\,.
 		\end{align*}
 		
 		Finally, by continuity this extends to all $a \in F^{p}_{\lambda}(\mathcal{G};\mathcal{E}_x)$.
 	\end{proof}
 	
 	\begin{lemma}
 		\label{help}
 		Let $a \in F^p_\lambda(\mathcal{G};\mathcal{E})$,  $x \in \mathcal{G}_0$ and $\gamma, \alpha \in \mathcal{E}_x$, then 
 		$$
 		\langle \lambda_x(a)(\delta_{\gamma})\,,\delta_\alpha \rangle  = \langle \lambda_{r(\gamma)}(a)(\delta_{r(\gamma)})\,,\delta_{\alpha\gamma^{-1}} \rangle \,.
 		$$
 		
 	\end{lemma}
 	\begin{proof}
 		let $f \in \Sigma_c(\mathcal{G};\mathcal{E})$, then
 		\begin{align*}
 			\langle \lambda_{x}(f)(\delta_{\gamma})\,,\delta_\alpha \rangle &= \sum_{\beta \in \mathcal{G}_x}\lambda_x(f)(\delta_\gamma)(S(\beta))\overline{\delta_\alpha(S(\beta))} \\
 			&= \lambda_x(f)(\delta_\gamma)(S(\pi(\alpha)))z(\alpha,S(\pi(\alpha))) \\
 			&= \lambda_x(f)(\delta_\gamma)(z(\alpha,S(\pi(\alpha)))\cdot S(\pi(\alpha))) \\
 			&= \lambda_x(f)(\delta_\gamma)(\alpha) \\
 			&= f(\alpha\gamma^{-1}) =  \langle\lambda_{r(\gamma)}(f)(\delta_{r(\gamma)})\,,\delta_{\alpha\gamma^{-1}} \rangle.
 		\end{align*}

 		By continuity this holds for all $a \in F^{p}_{\lambda}(\mathcal{G};\mathcal{E})$.
 	\end{proof}
 	
 	Given a Banach space $E$ and an operator $T \in \mathcal{B}(E)$ we write by $T^{\prime}$ the \emph{adjoint} of $T$. The operator $T^\prime \in \mathcal{B}(E^\ast)$, where $E^\ast$ is the dual of $E$,  is determined by $\langle x\,,T^\prime x^\ast \rangle = \langle Tx\,, x^\ast \rangle$ for all $x \in E$ and $x^\ast \in E^\ast$ where $\langle \cdot \,, \cdot \rangle$ represent the dual pairing of $E^\ast$ and $E$. Recall that $\norm{T} = \norm{T^\prime}$.
 	
 	\begin{lemma}
 		Let $x \in \mathcal{G}^{ \left(0 \right)}$. For $a \in F^{p}_{\lambda}(\mathcal{G};\mathcal{E})$, we write $\lambda_{x}(a)^{\prime} \colon L^q(\mathcal{G}_x;\mathcal{E}_x) \rightarrow L^q(\mathcal{G}_x;\mathcal{E}_x)$ for the adjoint of $\lambda_x(a)$. We define the contractive linear maps
 		$l_x \colon F^{p}_{\lambda}(\mathcal{G};\mathcal{E}) \rightarrow L^{p}(\mathcal{G}_x;\mathcal{E}_x)$  and $r_x \colon F^{p}_{\lambda}(\mathcal{G};\mathcal{E}) \rightarrow L^{q}(\mathcal{G}_x;\mathcal{E}_x)$ by 
 		\begin{equation*}
 			r_x(a) := \lambda_{x}(a)^{\prime}\delta_x\qquad \textrm{and}\qquad l_x(a) := \lambda_{x}(a)\delta_x\,,
 		\end{equation*}
 		for $a \in F^p_\lambda(\mathcal{G},\mathcal{E})$. Then $r_x(a)(\gamma) = \overline{j_a(\gamma^{-1})}$ and $l_x(a)(\gamma) = j_a(\gamma)$ for all $\gamma \in \mathcal{E}_{x}$.
 	\end{lemma}
 	
 	\begin{proof}
 		Let $x \in \mathcal{G}^{(0)}$ and let $\gamma \in \mathcal{E}_x$. For $a \in F^{p}_{\lambda}(\mathcal{G},\mathcal{E})$ we have by Proposition \ref{help} that
 		
 		$$
 		l_x(a)(\gamma) = \langle \lambda_{x}(a)(\delta_{x})\,,\delta_{\gamma} \rangle =  j_a(\gamma)\,.
 		$$
 		
 		Similarly, using Lemma 	\ref{help} we have that
 		$$
 		\overline{r_x(a)(\gamma)} = \langle \delta_{\gamma} \,,\lambda_{x}(a)^{\prime}\delta_x\rangle = \langle \lambda_{x}(a)(\delta_{\gamma})\,,\delta_x \rangle = \langle \lambda_{s(\gamma^{-1})}(a)(\delta_{s(\gamma^{-1})})\,,\delta_{{\gamma}^{-1}} \rangle = j_a(\gamma^{-1})\,,
 		$$
 		for all $\gamma \in \mathcal{E}_x$. It remains to show that the maps are contractive. We have that $\norm{l_x} = \sup \{\norm{l_x(a)}_{p} \colon \norm{a}_{\lambda} \leq 1\}$. Let $f \in \Sigma_c(\mathcal{G},\mathcal{E})$. Since $$
 		\norm{l_x(f)}_{p} = \|\lambda_x(f)(\delta_x)\|_p = \norm{f \ast \delta_x}_p\leq \norm{f}_\lambda\norm{\delta_x}_p=\norm{f}_\lambda,
 		$$
 		it follows that $\norm{l_x} \leq 1$. Similarly,  it follows that $r_x$ is contractive.
 	\end{proof}
 	
 	\begin{proposition}
 		\label{multi}
 		Let $a,b \in F^{p}_{\lambda}(\mathcal{G};\mathcal{E})$, $\gamma\in \mathcal{E}$. Then 
 		$$
 		j_{a \ast b}(\gamma) = \int_{\mathcal{E}_{s(\gamma)}} j_a(\gamma \alpha^{-1})j_b(\alpha)d\nu_{s(\gamma)}\,.
 		$$
 	\end{proposition}
 	
 	\begin{proof}
 		Let $x\in \mathcal{G}^{(0)}$. First we want to show that $\overline{j_a(\gamma \alpha^{-1})} = r_x(\delta_{\gamma^{-1}} \ast a)(\alpha)$ for all $\alpha \in \mathcal{E}_x$. Let $f \in \Sigma_c(\mathcal{G};\mathcal{E})$, then we have that
 		\begin{align*}
 			\overline{r_x(\delta_{\gamma^{-1}} \ast f)(\alpha)} &= \langle   \delta_{\alpha} \,, \lambda_x(\delta_{\gamma^{-1}} \ast f)^{\prime}(\delta_x) \rangle = \langle  \lambda_{x}(\delta_{\gamma^{-1}} \ast  f)(\delta_{\alpha}) \,, \delta_x  \rangle  \\
 			&= \langle  \lambda_{x}(\delta_{\gamma^{-1}}) \lambda_x  (f)(\delta_{\alpha}) \,, \delta_x  \rangle  = \langle   \lambda_x  (f)(\delta_{\alpha}) \,, \lambda_{x}(\delta_{\gamma^{-1}})'(\delta_x)  \rangle   \\
 			& = \langle   \lambda_x  (f)(\delta_{\alpha}) \,, \delta_\gamma  \rangle =  \langle   \lambda_{r(\alpha)}  (f)(\delta_{r(\alpha)}) \,, \delta_{\gamma\alpha^{-1}} \rangle = j_f(\gamma \alpha^{-1})\,.
 		\end{align*}
 		
 		Let $\gamma \in \mathcal{E}_x$, we now want to show that $\lambda_x(a)^\prime\delta_\gamma(\alpha) = \overline{j_a(\gamma\alpha^{-1})}$ for all $\alpha \in \mathcal{E}_x$. As usual, let $f \in \Sigma_c(\mathcal{G};\mathcal{E})$. Using Lemma \ref{help} in the third step, we have that
 		\begin{align*}
 			\overline{\lambda_x(a)^\prime \delta_\gamma(\alpha)} = \langle  \delta_{\alpha} \,, \lambda_x(a)^\prime(\delta_\gamma) \rangle =  \langle  \lambda_x(a)(\delta_{\alpha}) \,, \delta_\gamma  \rangle = \langle \lambda_{r(\alpha)}(a)(\delta_{r(\alpha)}) \,,\delta_{\gamma\alpha^{-1}} \rangle = j_a(\gamma\alpha^{-1})\,.
 		\end{align*}
 		
 		Combining the two statements we finally have 
 		\begin{align*}
 			j_{a \ast b}(\gamma) &= \langle \lambda_{x}(a \ast b)(\delta_{x})\,,\delta_\gamma \rangle = \langle \lambda_{x}(a)\left(\lambda_{x}(b)(\delta_{x})\right)\,,\delta_\gamma \rangle \\
 			&=  \langle  \lambda_{x}(b)(\delta_{x})\,,\lambda_{x}(a)^{\prime}\delta_\gamma \rangle = \langle l_{x}(b)\,, r_x(\delta_{\gamma^{-1}} \ast a)\rangle \\
 			& 	= \sum_{\beta \in \mathcal{G}_x} l_{x}(b)(S(\beta)) \overline{r_x(\delta_{\gamma^{-1}} \ast a)(\gamma S({\beta})^{-1})} \\
 			&= \sum_{\beta \in \mathcal{G}_x} j_b(S(\beta)) j_a(\gamma S({\beta})^{-1})=  \int_{\mathcal{E}_x} j_a(\gamma\alpha^{-1})j_b(\alpha)d\nu_x \,,
 		\end{align*}
 		which is well defined since the sum is absolutely convergent as it is given by the dual paring of elements in $L^{p}(\mathcal{G}_x;\mathcal{E}_x)$ and $L^{q}(\mathcal{G}_x;\mathcal{E}_x)$, and the last equality follows from (\ref{section_conv}). 
 	\end{proof}

 	\section{Rigidity of reduced twisted groupoid $L^p$-operator algebras.}
 	
 	A key ingredient to prove the rigidity results of groupoid $L^p$-operator algebras is the \emph{$C^\ast$-core}. The following results of the section come from  \cite[Sections 2 and 3]{choi2021rigidity}.

 	\subsection{$C^*$-cores of Banach algebras.}
 	Let $A$ be a unital Banach algebra. An element $a$ of $A$ is called \emph{hermitian} if $\norm{e^{ita}} = 1$ for all $t \in \mathbb{R}$. We denote the set of hermitian elements of $A$ by $A_h$, which is a closed real linear subspace satisfying $A_h \cap iA_h = \{ 0 \}$.
 	
 	If $A$ is a unital $C^\ast$-algebra, then $A_h$ consists of the self adjoint elements and $A = A_h + iA_h$. The Vidav–Palmer theorem shows that the converse also holds. So if $A$ is a unital Banach algebra with $A = A_h + iA_h$, then the real-linear involution $x+iy \mapsto x-iy$ is both isometric and an algebra involution that satisfies the $C^\ast$-identify. This motivates the following definition:
 	
 	\begin{definition}
 		Let $A$ be a unital Banach algebra, and let $B \subseteq A$ be a unital closed subalgebra. We say that $B$ is a \emph{unital $C^\ast$-subalgebra} of $A$ if $B = B_h + iB_h$.
 	\end{definition}
 	
 	The following theorem is \cite[Theorem 2.9]{choi2021rigidity} and will be the backbone of the rigidity result.
 	
 	\begin{theorem}
 		\label{core}
 		Let $p \in [1,\infty)$, and let A be unital $L^{p}$-operator algebra. Set $ \textrm{core}(A) = A_h + i A_h$. Then $\textrm{core}(A)$ is the largest unital $C^{\ast}$-subalgebra of $A$. If $p\neq 2$, then $\textrm{core}(A)$ is commutative.
 	\end{theorem}
 	
 	The last statement for $p\neq2$ is a consequence of Lamperti's Theorem.
 	
 	\begin{definition}
 		Let $p \in [1,\infty)$, and let $A$ be a unital $L^p$-operator algebra. We call the algebra $\textrm{core}(A)$  the \emph{$C^\ast$-core} of $A$.
 	\end{definition}
 	
 	The $C^\ast$-core will play the same role as the \emph{maximal abelian subalgebra} does for $C^\ast$-algebras, but there are  two important differences. Firstly, the $C^\ast$-core is unique, which will give rise to the rigidity result. Secondly, the $C^\ast$-core is very small, and in many cases it is too small to carry any useful information about the structure of the algebra. The last point is the reason why there will be only considered topologically principal groupoids. 
 	
 	\begin{proposition}{\cite[Proposition 2.13]{choi2021rigidity}}\label{homo_core}
 		\label{core-map}
 		Let $p \in [1,\infty)$, and let $A$ and $B$ be two unital $L^p$-operator algebras. Let $\varphi \colon A \rightarrow B$ be a unital contractive linear map. Then $\varphi(\text{core}(A)) \subseteq \textrm{core}(B)$ and $\varphi \colon \textrm{core}(A) \rightarrow \textrm{core}(B)$ is a $\ast$-homomorphism.
 	\end{proposition}
 	
 	Now let $A$ be a unital $L^{p}$-operator algebra, with $p \neq 2$, we  know from Theorem \ref{core} that the $C^\ast$-core is a commutative unital $C^{\ast}$-subalgebra. We will write $X_A$ for its spectrum. Recall that the \emph{spectrum} of an Banach algebra is the set of characters endowed with the $w^\ast$-topology. For a unital commutative Banach algebra the spectrum is a compact Hausdorff space, as is the case for $X_A$, and by the Gelfand transform we can isometrically identify the core with $C(X_A)$. 
 	
 	We can finally identify the $C^\ast$-core of $F^{p}_{\lambda}(\mathcal{G};\mathcal{E}$) following the proof of \cite[Proposition 5.1]{choi2021rigidity}. 
 	
 	\begin{theorem}
 		\label{rigid_core}
 		Let $p \in [1,\infty)\setminus\{2\}$, and let $\mathcal{E}$ be a twist over a locally compact étale groupoid $\mathcal{G}$ with compact unit space. Then $\text{core}(F^{p}_{\lambda}(\mathcal{G};\mathcal{E}))= C(\mathcal{G}^{(0)})$.
 	\end{theorem}
 	
 	\begin{proof}
 		Let $a \in \text{core}(F_{\lambda}^{p}(\mathcal{G};\mathcal{E}))$. With the map given in Proposition \ref{map} we want to  show that $\text{supp}(j_a) \subseteq i(\mathcal{G}^{(0)} \times \mathbb{T})$. Fix $x \in \mathcal{G}^{(0)}$. Then $\lambda_x \colon F_{\lambda}^{p}(\mathcal{G};\mathcal{E}) \rightarrow \mathcal{B}(L^p(\mathcal{G}_x;\mathcal{E}_x))$ is a contractive representation and it follows from Proposition \ref{homo_core} that $\lambda_x(\text{core}(F_{\lambda}^{p}(\mathcal{G};\mathcal{E}))) \subseteq \text{core}(\mathcal{B}(L^p(\mathcal{G}_x;\mathcal{E}_x)))$ and so by  \cite[Example 2.11]{choi2021rigidity} we have that $\lambda_x(a)$ is a multiplication operator in $\mathcal{B}(L^p(\mathcal{G}_x;\mathcal{E}_x))$. Let $\gamma \in \mathcal{E}$ and set $x = s(\gamma)$. Since $\lambda_x(a)$ is a multiplication operator, we have that $\lambda_x(a)\delta_x = c\delta_x$ for some constant $c \in \mathbb{T}$. Therefore, given $\gamma\notin i(\mathcal{G}^{(0)} \times \mathbb{T})$ we have that 
 		
 		\begin{align*}
 			j_a(\gamma) &= \langle \lambda_x(a) \delta_x \,, \delta_\gamma \rangle 
 			= \langle c\delta_x \,, \delta_\gamma \rangle 
 			= c\sum_{\beta \in \mathcal{G}_x} \delta_x(S(\beta))\overline{\delta_\gamma(S(\beta))}=0\,,
 		\end{align*}
 		where $S\colon\mathcal{G}\to \mathcal{E}$ is a section of $\pi$.
 		It follows that $\text{supp}(j_a) \subseteq i(\mathcal{G}^{(0)} \times \mathbb{T})$. By Proposition \ref{multi} we have that $j$ is a homomorphism. By Proposition \ref{map} the map $j$ is injective, so $a=j_a$, and using the identity (\ref{core-id}) we have that $j$ induces a $^\ast$-isomorphism $\text{core}(F^{p}_{\lambda}(\mathcal{G};\mathcal{E})) \cong C(\mathcal{G}^{(0)})$.
 	\end{proof}
 	
 	\begin{notation}
 		We will for the rest of the paper omit the the map $j$ from notation and just write $a(\gamma)$ for $j_a(\gamma)$, and we will fix a section $S\colon\mathcal{G}\to \mathcal{E}$.
 	\end{notation}

 	\subsection{Admissible pairs and the Weyl groupoid.}

 	We will denote the subset of continuous non-negative real functions by $C(X_A)_+$, this is the set of positive hermitian elements in $A$.
 	\begin{definition}
 		\label{normal}
 		Let A be a unital $L^{p}$-operator algebra. Given open subsets $U,V \subseteq X_A$, and a homeomorphism $\alpha \colon U \rightarrow V$, we say that $\alpha$ is \emph{realizable}  (within $A$) if there exist $a,b \in A$ satisfying the following:
 		\begin{enumerate}
 			\item For every $f \in C(X_A)_+$, we have that $bfa ,afb \in C(X_A)_+$,
 			\item $U = \{x \in X_A \colon ba(x) >0\}$ and $V = \{x \in X_A \colon ab(x) >0\}$, 
 			\item for $x \in U$, $y \in V$, $f \in C_0(V)$ and $g \in C_0(U)$ we have that
 			$$
 			f(\alpha(x))ba(x) = bfa(x) \textrm{  and  } g(\alpha^{-1}(y))ab(y) = agb(y)\,.
 			$$
 		\end{enumerate}
 		We then say that $\mathrm{n} = (a,b)$ is an \emph{admissible pair} that \emph{realizes} $\alpha$, and we write $\alpha_\mathrm{n}$, $U_\mathrm{n}$ and $V_\mathrm{n}$ for $\alpha$, $U$ and $V$ respectively.
 	\end{definition}

 	The following lemma is straightforward to prove.
 	\begin{lemma}
 		Let A be a unital $L^{p}$-operator algebra and let $\textrm{n}=(a,b)$ be an admissible pair, then given $z\in \mathbb{C}$ the pair $z\cdot \textrm{n}:=(za,\overline{z}b)$ is an admissible pair.
 	\end{lemma}

 	The admissible pairs  will play the role as the $L^p$-analogue of the normalizers used by Renault in \cite{renault2008cartan} in the context of Cartan pairs of $C^\ast$-algebras, i.e., the pair $(a,b)$ replaces the pair $(a,a^\ast)$ where $a$ is a normalizer. In our setting of $L^p$-operator algebras, there are a number of difficulties arising from the lack of canonical involution.
 	First we recall some of the properties of admissible pairs shown in  \cite[Proposition 3.2]{choi2021rigidity}. 
 	\begin{proposition}
 		\label{norm_prop}
 		Let $p \in [1,\infty)\setminus\{2\}$, let $A$ be a unital $L^{p}$-operator algebra and let $\mathrm{n} = (a,b)$ and $\mathrm{m} = (c,d)$ be two admissible pairs in $A$ that realize $\alpha_\mathrm{n}$ and $\alpha_\mathrm{m}$, respectively.
 		\begin{enumerate}
 			\item The inverse of $\alpha_\mathrm{n}$ is realized by the reverse of $\mathrm{n}$ which is defined as $\mathrm{n}^{\sharp} = (b,a)$.
 			\item The product $\mathrm{n}\mathrm{m} = (ac,db)$ realizes the composition
 			$$
 			\alpha_\mathrm{n} \circ \alpha_\mathrm{m} \restriction_{U_\mathrm{m} \cap \alpha_\mathrm{m}^{-1}(U_\mathrm{n})} \colon U_\mathrm{m} \cap \alpha_\mathrm{m}^{-1}(U_\mathrm{n}) \rightarrow U_\mathrm{n} \cap \alpha_\mathrm{n}(V_\mathrm{m})\,.
 			$$
 			\item For every $f \in C(X_A)$, the pair $\mathrm{n}_f = (f,\overline{f})$ is admissible and realizes $\alpha_{\mathrm{n}_f}= Id_{U}$ for $U=\textrm{supp}(f)$. In particular, the identity map on every open set of $X_A$ is realizable. 
 		\end{enumerate}
 	\end{proposition}

 	The following corollary follows from Proposition \ref{norm_prop}.
 	
 	\begin{corollary}
 		Let $A$ be a unital $L^{p}$-operator algebra, then the set of realizable partial homeomorphism on $X_A$, denoted by $N(A)$, is an inverse subsemigroup of $\textrm{Homeo}_{par}(X_A)$. 
 	\end{corollary}

 	\begin{definition}
 		Let $A$ be a unital $L^p$-operator algebra. We define the \emph{Weyl groupoid of $A$}, denoted by $\mathcal{G}_A$, to be the groupoid of germs of $N(A)$. 
 	\end{definition}
 	
 	\subsection{The Weyl groupoid of $F_\lambda^p(\mathcal{G};\mathcal{E})$.}
 	Next we will show the relationship between two classes of partial homeomorphisms on $\mathcal{G}^{(0)}$, the ones induced by open bisections and the ones realized by admissible pairs of $F_\lambda^p(\mathcal{G};\mathcal{E})$.

 	Given a topological space $X$, let $U \subseteq X$ be an open set, we say $U$ is a  cozero set if there exists a continuous function $f \colon X \rightarrow \mathbb{C}$ such that $U = \textrm{supp}(f)$
 	
 	\begin{proposition}
 		\label{admissible}
 		Let $p \in [1,\infty)\setminus\{2\}$, and let $\mathcal{E}$ be a twist over a locally compact, Hausdorff, étale groupoid $\mathcal{G}$ with compact unit space. Let $B$ be an open bisection of $\mathcal{G}$ with associated partial homeomorphism $\beta_B \colon s(B) \rightarrow r(B)$. Let $U \subseteq \mathcal{G}^{(0)}$ be a cozero set. Then the restriction of $\beta_B$ to U is realizable by an admissible pair in $F^p_\lambda(\mathcal{G};\mathcal{E})$.
 	\end{proposition}

 	\begin{proof}
 		We can without loss of generality replace $B$ with $\{\gamma \in B \colon s(\gamma) \in U\}$ and assume $s(B) \subseteq  U$. We then need to prove that $\beta_B$, given as in Remark \ref{beta}, is realizable by an admissible pair in $F^p_\lambda(\mathcal{G};\mathcal{E})$. Let $h \in C(\mathcal{G}^{(0)})$ be any function such that $\textrm{supp}(h) = s(B)$. Choose a non-vanishing continuous $\mathbb{T}$-equivariant function $u \colon \pi^{-1}(B) \rightarrow \mathbb{C}$. Replace $u(\gamma)$ with $u(\gamma) / |u(\gamma)|$ and we may assume that $|u(\gamma)| = 1$. Define the functions $a,b \colon \mathcal{E} \rightarrow \mathbb{C}$ as 
 		\begin{equation*}
 			a(\alpha) =
 			\begin{cases} 
 				u(\alpha)h(s(\alpha)) & \textrm{if}\, \pi(\alpha) \in B \\
 				0 & otherwise\,, \\
 			\end{cases} 
 			\qquad\text{and}\qquad
 			b(\alpha) =
 			\begin{cases} 
 				\overline{u(\alpha^{-1})h(r(\alpha))} & \textrm{if}\, \pi(\alpha^{-1}) \in B \\
 				0 & otherwise\,, \\
 			\end{cases}
 		\end{equation*}
 		for all $\alpha\in \mathcal{E}$.
 		The functions $a$ and $b$ are $\mathbb{T}$-equivariant. Since both functions $a,b$ have support on $\pi^{-1}(B)$ it means the $I$-norm is equal to the $\infty$-norm for $a$ and $b$ by Lemma \ref{norm}. It follows that $a,b \in F^p_\lambda(\mathcal{G},\mathcal{E})$ since they are $I$-norm limits of elements of elements in $C_c(\pi^{-1}(B))$.

 		We then want to show that $\textrm{n}= (a,b)$ is the admissible pair that realizes $\beta_B$. First we will prove that $\textrm{n}$ satisfy the axioms in Definition \ref{normal}. Let $f \in  C(\mathcal{G}^{(0)})_{+}$ and let $\gamma\in \mathcal{E}$. Then 
 		$$
 		bfa(\gamma) = \int_{\mathcal{E}_{s(\gamma)}} b(\gamma\alpha^{-1})f(r(\alpha))a(\alpha)d\nu_{s(\alpha)}\,.
 		$$
 		Now, if $bfa(\gamma) \neq 0$, then there exists an element $\alpha \in \mathcal{E}_{s(\gamma)}$ such that $b(\gamma\alpha^{-1})f(r(\alpha))a(\alpha) \neq 0$, then $\pi((\gamma\alpha^{-1})^{-1}) = \pi(\alpha\gamma^{-1}) \in B$ and $\pi(\alpha) \in B$. Since $B$ is an open bisection, the range map restricted to $B$ is injective. This then implies that $\gamma^{-1} = s(\alpha) \in \mathcal{G}^{(0)}$. Thus it follows that $\textrm{supp}(bfa) \subseteq i(\mathcal{G}^{(0)}\times \mathbb{T})$ which means that $bfa \in C(\mathcal{G}^{(0)})$. We then need to show that $bfa$ is positive. Since 
 		$$
 		bfa(x) = \int_{\mathcal{E}_x} b(\alpha^{-1})f(x)a(\alpha)d\nu_x
 		= \int_{\mathcal{E}_x} |h(x)|^2f(x)d\nu_x,
 		$$
 		it follows that $bfa(x) > 0$ for all $x \in \mathcal{G}^{(0)}$, and hence $bfa \in C(\mathcal{G}^{(0)})_+$. Similarly one can show that $afb \in C(\mathcal{G}^{(0)})_+$. Thus the first condition of Definition \ref{normal} holds. 
 		
 		Now, to show the second condition note that from the first condition, it follows that $ab,ba \in C(\mathcal{G}^{(0)})_{
 			+}$. Let $x \in s(B)$. Since $B$ is an open bisection, let $\alpha_0 \in B$ be the unique element such that $s(\alpha_0) = x$ and let $S(\alpha_0) \in \pi^{-1}(\alpha_0)$. Then
 		
 		$$
 		ba(x) = \int_{\mathcal{E}_x}b(\alpha^{-1})a(\alpha)d\nu_{x} =  b(S(\alpha_0)^{-1})a(S(\alpha_0)) = |h(x)|^{2}\,.
 		$$
 		It follows that $ba(x) = |h(x)|^2$. On the other hand if $x \in \mathcal{G}^{(0)} \setminus{s(B)}$, then 
 		$$
 		ba(x) =  \int_{\mathcal{E}_x}b(\alpha^{-1})a(\alpha) d\nu_{x} = 0\,,
 		$$
 		since $b(\alpha^{-1})$ and  $a(\alpha)$ are zero for all $\alpha \in \mathcal{E}_x$. Thus, $ba = |h|^2$, and it follows that $s(B) = \textrm{supp}(h) = \textrm{supp}(ba)$.
 		Similarly, for $x \in s(B)$ we can show that $ab(\beta_B(x))=|h(x)|^2$, and for $x \in \mathcal{G}^{(0)} \setminus{s(B)}$ that $ab(\beta_B(x))= 0$, and so $|h|^2 = |ab \circ \beta_B|$. Thus, $r(B) = \beta_B(s(B)) = \textrm{supp}(ab)$.
 		
 		Finally, to show the third condition, let $x \in s(B)$ and let $f \in C_0(r(B))$. Since $B$ is an open bisection, there is a unique element $\alpha_0 \in B$ such that $s(\alpha_0) = x$. Then we have 
 		\begin{align*}
 			bfa(x) &= \int_{\mathcal{E}_{x}} b(\alpha^{-1})f(r(\alpha))a(\alpha)d\nu_x = \sum_{\gamma\in \mathcal{G}_{x}} b(S(\alpha)^{-1})f(r(\alpha))a(S(\alpha)) \\
 			&= b(S(\alpha_0)^{-1})f(r(\alpha_0))a(S(\alpha_0)) =  f(\beta_B(x))|h(x)|^2 = f(\beta_B(x))ba(x)\,.
 		\end{align*}
 		Similarly, for $y \in r(B)$ and $g \in C_0(s(B))$ it follows that $g(\beta_B^{-1}(y))ab(y) = agb(y)$. Hence $\textrm{n}=(a,b)$ is the admissible pair that realizes $\beta_B$, and the result follows.
 	\end{proof}
 	
 	\begin{remark}
 		If $\mathcal{G}^{(0)}$ is metrizable, then every open set is a cozero set. In general we have that $\mathcal{G}^{(0)}$ is a compact Hausdorff space so we have that for every $x \in \mathcal{G}^{(0)}$ there exists an open cozero neighborhood $U$ of $x$.
 	\end{remark}
 	
 	The next proposition makes the "converse" relation. We find an open bisection given an admissible pair.
 	
 	\begin{proposition}
 		\label{bisection}
 		Let $p \in [1,\infty)\setminus\{2\}$, and let $\mathcal{E}$ be a twist over a topologically principal, locally compact, Hausdorff, étale groupoid $\mathcal{G}$ with compact unit space. Let $\textrm{n}=(a,b)$ be an admissible pair in $F^p_\lambda(\mathcal{G}; \mathcal{E})$. Set 
 		$$
 		B_\textrm{n} = \{\gamma \in \mathcal{G} \colon a(S(\gamma)),b(S(\gamma)^{-1}) \neq 0\}.
 		$$
 		Then $B_\textrm{n}$ is an open bisection of $\mathcal{G}$ and $\alpha_\textrm{n} = \beta_{B_\textrm{n}}$.
 	\end{proposition}
 	
 	We will call $B_\textrm{n}$ \emph{the open bisection induced by} $\textrm{n}$.
 	
 	\begin{proof}
 		First note that if $a(\gamma) = 0$, then $a(\varepsilon) = 0$ for all $\varepsilon \in \mathcal{E}$ with $\pi(\gamma) = \pi(\varepsilon)$.
 		
 		%Choose $f \in C(\mathcal{G}^{(0)})$ with $f(\gamma)$ = 1 for all $\gamma \in \mathcal{G}^{(0)}$. 
 		
 		Since $(a,b)$ is an admissible pair we have that $ba\in C_0(\mathcal{E})_+$  with  $\text{supp}(ba) \subseteq i(\mathcal{G}^{(0)} \times \mathbb{T})$. Observe that $b(\gamma^{-1})a(\gamma) = b(\alpha^{-1})a(\alpha)$ for any $\alpha ,\gamma \in \mathcal{E}$ with $\pi(\alpha)=\pi(\gamma)$. Then given $x\in \mathcal{G}^{(0)}$ we have that
 		
 		$$
 		ba(x)= \sum_{\gamma \in \mathcal{G}_x} b(S(\gamma)^{-1})a(S(\gamma)) \geq 0 \,. 
 		$$
 		We claim that $b(S(\gamma)^{-1})a(S(\gamma)) \geq 0$ for any $\gamma \in \mathcal{G}$. First, assume that there exists $\gamma\in \mathcal{G}$ such that $\textrm{Re}(b(S(\gamma)^{-1})a(S(\gamma))) < 0$. 
 		Since $a$ and $b$ are continuous, there is an open neighborhood $U$ of $S(\gamma)$ such that $\textrm{Re}(b(\xi^{-1})a(\xi))<0$ for all $\xi \in U$. Set $V = s(U)$, which is an open subset of $\mathcal{G}^{(0)}$. Since $\mathcal{G}$ is topological principal, there is a $x_0 \in V$ with trivial isotropy. Let $\gamma_0 \in \pi(U)$ with $s(\gamma_0) = x_0$, and set $y=r(\gamma_0)$. Then, since $x_0$ has trivial isotropy, $\gamma_0$ is the unique element in $\mathcal{G}_{x_0}$ with $y$ as it range. Since
 		$$
 		ba(x_0)=\textrm{Re}(ba(x_0)) = \int_{\mathcal{E}_{x_0}} \textrm{Re}(b(\xi^{-1})a(\xi))d\nu_{x_0} =  \sum_{\eta\in \mathcal{G}_{x_0}} \textrm{Re}(b(S(\eta)^{-1})a(S(\eta)))
 		$$ 
 		converges absolutely, the set $\{\eta\in \mathcal{G}_{x_0} \colon \textrm{Re}(b(S(\eta)^{-1})a(S(\eta)))\neq 0 \}$ is at most countable. Set $t = \textrm{Re}(b(S(\gamma_0)^{-1})a(S(\gamma_0)))< 0$, and choose a neighborhood $W$ of $y$ such that
 		\begin{equation}\label{inequa}
 			\sum_{\eta \in \mathcal{G}_{x_0},r(\eta) \in W\setminus\{y\}} |\textrm{Re}(b(S(\eta)^{-1})a(S(\eta)))| < |t| = -t \,.
 		\end{equation}
 		Choose $f \in C_0(\mathcal{G}^{(0)})_+$ with $0\leq f \leq 1$, $f(y) = 1$ and $\textrm{supp}(f) \subseteq W$. Then using (\ref{inequa}) we have that
 		\begin{align*}
 			bfa(x_0) &= \int_{\mathcal{E}_{x_0}} \textrm{Re}(b(\xi^{-1})f(r(\xi))a(\xi))d\nu_{x_0} = \sum_{\eta \in \mathcal{G}_{x_0}} \textrm{Re}(b(S(\eta)^{-1})f(r(\eta))a(S(\eta))) \\
 			&= \sum_{\eta \in \mathcal{G}_{x_0},r(\eta) \in W\setminus\{y\}}  \textrm{Re}(b(S(\eta)^{-1})f(r(\eta))a(S(\eta))) + \textrm{Re}(b(S(\gamma_0)^{-1})a(S(\gamma_0)))  \\
 			& \leq \sum_{\eta \in \mathcal{G}_{x_0},r(\eta) \in W\setminus\{y\}}  |\textrm{Re}(b(S(\eta)^{-1})a(S(\eta)))| + t <0 \,,
 		\end{align*}
 		which contradicts condition (1) of Definition \ref{normal}, and thus we have that $\textrm{Re}(b(S(\gamma)^{-1})a(S(\gamma))) \geq 0$ for all $\gamma \in \mathcal{G}$. Now in a similar way we can prove that $\textrm{Im}(b(S(\gamma)^{-1})a(S(\gamma))) \geq 0$ for all $\gamma \in \mathcal{G}$, but since 
 		$$0=\textrm{Im}(ba(x))=\sum_{\gamma\in \mathcal{G}_x}\textrm{Im}(b(S(\gamma)^{-1})a(S(\gamma)))\,,$$ 
 		for every $x\in\mathcal{G}^{(0)}$ we have  that $\textrm{Im}(b(S(\gamma)^{-1})a(S(\gamma))) =0$ for all $\gamma \in \mathcal{G}$. Thus, $b(S(\gamma)^{-1})a(S(\gamma))\geq 0$ for every $\gamma\in \mathcal{G}$, as desired.
 		
 		Let us denote 
 		$$B = \{\gamma \in \mathcal{G} \colon a(S(\gamma)),b(S(\gamma)^{-1}) \neq 0\}\,.$$
 		Let $\gamma \in B$. We want to prove that $s(\gamma) \in U_\textrm{n}$ and $r(\gamma) \in V_\textrm{n}$. Set $x = s(\gamma)$. Then we have that 
 		$$
 		ba(x) = \sum_{\xi \in Bx} b(S(\xi)^{-1})a(S(\xi)) > b(S(\gamma)^{-1})a(S(\gamma)) > 0
 		$$
 		which by (2) in Definition \ref{normal} implies that $x \in U_\textrm{n}$. Similarly it follows that $y = r(\gamma) \in V_\textrm{n}$.

 		Let $\gamma_0 \in B$ and set $x = s(\gamma_0)$. We want to show that $r(\gamma_0) = \alpha_\textrm{n}(x)$. Assume that $r(\gamma_0) \neq \alpha_\textrm{n}(x)$. Choose $f \in C(V_\textrm{n})_+$, with $f(\alpha_\textrm{n}(x)) = 0$ and $f(r(\gamma_0)) = 1$, then
 		\begin{align*}
 			f(\alpha_\textrm{n}(x)) &= \frac{bfa(x)}{ba(x)} = \int_{\mathcal{E}_{x}} \frac{ b(\xi^{-1})f(r(\xi))a(\xi)}{ba(x)}d\nu_{s(\gamma)} \\
 			&= \sum_{\gamma \in \mathcal{G}_{x}} \frac{ b(S(\gamma)^{-1})f(r(\gamma))a(S(\gamma))}{ba(x)} >  \frac{b(S(\gamma_0)^{-1})a(S(\gamma_0))}{ba(x)} > 0 \,.
 		\end{align*}

 		Define the set
 		$T = \{ \gamma \in \mathcal{G} \colon s(\gamma) \in U_n\text{ and }r(\gamma) = \alpha_n (s(\gamma))\}$. We have already shown that $B \subseteq T$, which implies that $BB^{-1} \subseteq TT^{-1} \subseteq \mathcal{G}^{\prime} = \{\gamma\in \mathcal{G} \colon s(\gamma) = r(\gamma) \}$. We will show that $BB^{-1}$ is open. Let $i \colon \mathcal{E} \rightarrow \mathcal{E}$ denote the inversion map. Since $a$ and $b \circ i$ are continuous functions on $\mathcal{E}$, we know that their supports are open subsets of $\mathcal{E}$. We then have that $B=\pi(\textrm{supp}(a)\cap\textrm{supp}(b\circ i))$ is open since $\pi$ is an open map. Since $\mathcal{G}$ is an effective groupoid we have that $BB^{-1}$ is an open set contained in $\mathcal{G}^{(0)}$. Similarly one can show that $B^{-1}B$ is an open set contained in $\mathcal{G}^{(0)}$ . It follows that $B$ is an open bisection.
 		
 		Fix $x \in U_{\textrm{n}}$. We want to show that there exists $\gamma_0 \in B$ with $s(\gamma_0) = x$. For any $x \in U_{\textrm{n}}$ we have that
 		$$
 		0 < ba(x) = \sum_{\gamma \in \mathcal{G}_x} b(S(\gamma))a(S(\gamma)^{-1})\,.
 		$$
 		This implies that there exists $\gamma_0 \in \mathcal{G}_x$ with $b(S(\gamma_0))a(S(\gamma_0)^{-1}) > 0$, which means that $\gamma_0 \in B$ and $s(\gamma_0) = x$.
 		
 		We thus have that $s(B) = U_{\textrm{n}}$. It remains to show that $\alpha_{\textrm{n}} = \beta_B$. Let $x \in U_{\textrm{n}}$, and let $\gamma \in B$ be the element such that $s(\gamma) = x$. Then $\beta_B(x) = r(\gamma) = \alpha_{\textrm{n}}(x)$, and the proposition follows.
 	\end{proof}
 	
 	\begin{theorem}
 		\label{groupoid_rec}
 		Let $p \in [1,\infty)\setminus\{2\}$, let $\mathcal{E}$ a twist over a topologically principal, locally compact,  Hausdorff, étale groupoid $\mathcal{G}$ with compact unit space. Then there is a natural identification of groupoids
 		$$
 		\mathcal{G}_{F^p_\lambda(\mathcal{G};\mathcal{E})} \cong \mathcal{G}.
 		$$
 	\end{theorem}

 	\begin{proof}
 		By Theorem \ref{rigid_core} we identify the core of $F^p_\lambda(\mathcal{G};\mathcal{E})$ with $C(\mathcal{G}^{(0)})$. Let $\mathcal{A}$ be the set of all partial homeomorphisms realized by admissible pairs in $F^p_\lambda(\mathcal{G};\mathcal{E})$ and let $\mathcal{B}$ be the family of partial homeomorphisms on  $\mathcal{G}^{(0)}$ induced by open bisections of $\mathcal{G}$. By Proposition \ref{bisection} we have that $\mathcal{A} \subseteq \mathcal{B}$. The converse holds if every open subset of $\mathcal{G}^{(0)}$ is a cozero set, which is not the case in general, but it holds locally, i.e., for every $x \in \mathcal{G}^{(0)}$ there exists a cozero neighborhood $U$ of $x$. Using this fact we have by Proposition \ref{admissible} that for every $\beta \in \mathcal{B}$ and $x \in \mathcal{G}^{(0)}$, there exists an open neighborhood $U$ of $x$ such that $\beta_{\restriction_U} \in \mathcal{A}$. It follows that the groupoids of germs of $\mathcal{A}$ and $\mathcal{B}$ are isomorphic. The groupoids of germs of $\mathcal{A}$ is by definition $\mathcal{G}_{F^p_\lambda(\mathcal{G};\mathcal{E})}$. By \cite[Corollary 3.3]{renault2008cartan}, the groupoid of germs of $\mathcal{B}$ is isomorphic to $\mathcal{G}$, since $\mathcal{G}$ is effective.
 	\end{proof}
 	
 	\begin{remark}
 		\label{exp}
 		Explicitly the isomorphism $\theta\colon \mathcal{G}_{F^p_\lambda(\mathcal{G};\mathcal{E})} \rightarrow \mathcal{G}$ is given as follows. For $[\alpha_{\textrm{n}},x] \in \mathcal{G}_{F^p_\lambda(\mathcal{G};\mathcal{E})}$ we let $B_n$ denote the open bisection of the admissible pair $n$ as in Proposition \ref{bisection}. Then the isomorphism is given by $[\alpha_n,x] \mapsto B_n x$. See the proof of  \cite[Corollary 3.4]{renault2008cartan} for details.
 	\end{remark}
 	
 	We observe that the groupoid of germs does not recover the twist, and we only recover the groupoid as in \cite[Theorem 5.5]{choi2021rigidity}.

 	\subsection{The Weyl twist and $L^p$-rigidity}
 	
 	In the previous section we recovered the groupoid from the algebra. In this section we will see how to recover the twist $\mathcal{E}$ over $\mathcal{G}$ from the algebra. We will follow Renault's line of thought in \cite{renault2008cartan}, where he recovers the twist in the $C^\ast$-algebra setting.
 	
 	Let $A$ be a unital $L^{p}$-operator algebra for $p\in[1,\infty)\setminus\{2\}$, and let $N(A)$ denotes the inverse subsemigroup of realizable partial homeomorphisms of $X_A$. We define
 	$$
 	\mathcal{E}_A:=\{(\textrm{n},x) \in N(A) \times X_A \colon \textrm{n} = (a,b) \in N(A), ba(x)>0\}/ \approx\,,
 	$$
 	where $(\textrm{n},x) \approx (\textrm{m},y)$ whenever $x = y$ and there exist $f,g \in C(X_A)$ with $f(x),g(x) > 0$ such that $\textrm{n}_f \textrm{n} = \textrm{n}_g\textrm{m}$. We denote $[[\textrm{n},x]]$ for the equivalence class of the element $(\textrm{n},x)$. $\mathcal{E}_A$ has a natural groupoid structure. The range and source map are given by $r([[ \textrm{n},x ]]) = \alpha_\textrm{n}(x)$ and $s([[\textrm{n},x ]]) = x$. Multiplication is given by $[[\textrm{n},\alpha_\textrm{m}(x)]][[\textrm{m},x]] = [[\textrm{nm},x ]]$, and inversion  by $[[\textrm{n},x ]]^{-1} = [[\textrm{n}^\sharp,\alpha_\textrm{n}(x) ]]$. The unit space is canonically identified with $X_A$. The quotient becomes a topological groupoid under the topology with basic open sets $\mathcal{U}(\textrm{n},U,V) = \{ [[ z\cdot \textrm{n},x ]]  \colon x \in U\,,z\in V\}$ indexed over $\textrm{n} \in N(A)$ and open sets $U \subseteq U_\textrm{n}$ where $U_\textrm{n}$ is given as in Definition \ref{normal}, and $V\subseteq \mathbb{T}$. We call $\mathcal{E}_A$ the \emph{Weyl twist} of $A$, analogous to Renault´s in \cite{renault2008cartan}.

 	First we will show that multiplication is well defined. Let $(\textrm{n},x) \approx (\textrm{n}^\prime,x)$ and $(\textrm{m},\alpha_\textrm{n}(x)) \approx (\textrm{m}^{\prime},\alpha_\textrm{n}(x))$, where $\textrm{n} = (a,b), \textrm{n}^\prime = (a^\prime, b^\prime), \textrm{m} = (c,d)$ and $\textrm{m}^\prime = (c^\prime, d^\prime)$. Then there exist $f,f^\prime,g,g^\prime \in C(X_A)$ such that $\textrm{n}\textrm{n}_f = \textrm{n}^{\prime}\textrm{n}_{f^\prime}$ and $\textrm{m}\textrm{n}_g = \textrm{m}^\prime \textrm{n}_{g^\prime}$. We want to show that $(\textrm{n}\textrm{m},x) \approx (\textrm{n}^{\prime}\textrm{m}^{\prime},x)$. Note that $\textrm{n}_f$ realizes the identity map on some set $U_{\textrm{n}_f} \subseteq \mathcal{G}^{(0)}$.
 	Since $fa(\gamma) = a(\gamma)f(r(\gamma)) = a(f\circ \alpha_\textrm{m})(\gamma)$, we have that $\textrm{n}_f\textrm{m} = \textrm{m}\textrm{n}_{f\circ \alpha_\textrm{m}}$. Using this fact we have the following:
 	$$
 	\textrm{n}\textrm{n}_f\textrm{m}\textrm{n}_g = \textrm{n}\textrm{m}\textrm{n}_{f\circ \alpha_\textrm{m}}\textrm{n}_g = \textrm{n}\textrm{m}\textrm{n}_{(f\circ \alpha_\textrm{m})g}\,,
 	$$
 	and so
 	$$
 	\textrm{n}\textrm{m}\textrm{n}_{(f\circ \alpha_\textrm{m})g} = \textrm{n}^{\prime}\textrm{m}^{\prime}\textrm{n}_{(f^\prime \circ \alpha_\textrm{m})g^\prime}\,.
 	$$
 	
 	It thus follows that $(\textrm{n}\textrm{m},x) \approx (\textrm{n}^{\prime}\textrm{m}^{\prime},x)$, and hence multiplication is well defined. To show that the groupoid is a topological groupoid, let $U \subseteq \mathcal{G}^{(0)}$, then
 	$$
 	s^{-1}(U) = \bigcup_{\textrm{n}\in N(A) \colon U\cap U_\textrm{n} \neq Ø} \{ [[ \textrm{n},x ]] \colon x \in U\cap U_\textrm{n} \}\,,
 	$$
 	which is union of basic open sets and thus open. It follows that the source map is continuous.

 	\begin{proposition}
 		Let $A$ be a unital $L^p$-operator algebra for $p\in[1,\infty)\setminus\{2\}$. There is an injective continuous groupoid homomorphism $i_A \colon X_A \times \mathbb{T} \rightarrow \mathcal{E}_A$ given by $i_A((x,z)) = [[ \text{n}_f, x ]]$ for $f \in C(X_A)$ such that $f(x) = z$, and there is a continuous surjective groupoid homomorphism $\pi_A \colon \mathcal{E}_A \rightarrow \mathcal{G}_A $ given by $\pi_A([[ \textrm{n}, x ]]) = [\alpha_\textrm{n},x]$.
 	\end{proposition}
 	
 	\begin{proof}
 		Let $x \in X_A$, $z \in \mathbb{T}$, and let $f,g \in C(X_A)$ with $f(x) = g(x) = z$. We then have that $(\textrm{n}_f,x) \approx (\textrm{n}_g,x)$ since $\overline{z}f(x),\overline{z}g(x)>0$ and $f\overline{z}g = g\overline{z}f$, and so $\textrm{n}_f \textrm{n}_{\overline{z}g} = \textrm{n}_g \textrm{n}_{\overline{z}f}$. Therefore, $i_A$ is  well defined.
 		
 		Let $x,y \in X_A$, let $z,w \in \mathbb{T}$, and let $f,g \in C(X_A)$ with $f(x) = z$ and $g(y) = w$. If $(\textrm{n}_f,x) \approx (\textrm{n}_g,y)$, then $x = y$, and there exist $f^\prime,g^\prime \in C(X_A)$ with $f'(x),g'(x)>0$ such that $\textrm{n}_f \textrm{n}_{f^\prime} = \textrm{n}_g \textrm{n}_{g^\prime}$. Therefore,  $ff^\prime = gg^\prime$ which implies that $zf^\prime(x) = wg^\prime(x)$. But since $f^\prime(x), g^\prime(x) > 0$, this implies that $z=w$, and so $i_A$ is injective.

 		Now, choose an open set $W$ of $\mathcal{E}_A$ of the form $\{[[ z\cdot \textrm{n},x ]] \colon x \in U,z\in V\}$ for some $\textrm{n}=(a,b)\in N(A)$ and some open neighborhood $U$ contained in $\textrm{supp}(f)$ and open subset $V\subseteq \mathbb{T}$. Then 
 		$$i^{-1}_A(W)=\{(x,z)\in X_A\times \mathbb{T}:[[\textrm{n}_h,x]]\in W\text{ for some }h\in C(X_A)\text{ with }h(x)=z \}\,.$$
 		Observe that if $[[\textrm{n}_h,x]]\in W$ for some $h\in C(X_A)$ with $h(x)\in \mathbb{T}$, this means that there exist $f,g\in C(X_A)$ with $f(x),g(x)>0$ and $z\in V$ with $\textrm{n}_f(z\cdot \textrm{n})=\textrm{n}_g\textrm{n}_h$. Therefore, we have that  $za(x)f(x)=b(x)h(x)$ and $\overline{z}b(x)\overline{f(x)}=\overline{b(x)h(x)}$ for every $x\in (\text{supp }g)\cap (\text{supp }h)$ so $a(x)=\overline{b(x)}$ for every $x\in (\text{supp }g)\cap (\text{supp }h)$. Moreover,  since $(\text{supp }g)\cap (\text{supp }h)$ is an open subset of $X_A$ that contains $x$ and with $\alpha_\textrm{n}(y)=y$ for every $y\in (\text{supp }g)\cap (\text{supp }h)$, we have that $x\in (X_A\setminus  \overline{\{y\in X_A: \alpha_\textrm{n}(y)\neq y \}})\cap U=:E$. Observe that $E$ is an open subset of $X_A$. Then for every $x\in E$, $z\in V$ and $h\in C(X_A)$ with $\text{supp }h\subseteq E$, then $(z\cdot \textrm{n},x)\approx (\textrm{n}_h,x)$ if and only if $h(x)=za(x)/|a(x)|$. Therefore, $i^{-1}_A(W)=\{(x,za(x)/|a(x)|): x\in E\,, z\in V\}$  is an open subset of $X_A\times \mathbb{T}$, as desired.

 		Now let $(\textrm{n},x) \approx (\textrm{m},x)$, we want to show that $(\alpha_\textrm{n},x) \sim (\alpha_\textrm{m},x)$. Since $(\textrm{n},x) \approx (\textrm{m},x)$ we have that there exist $f,f^{\prime} \in C(X_A)$ with $f(x),f^\prime(x) > 0$ such that $\textrm{n}\textrm{n}_f = \textrm{m}\textrm{n}_{f^\prime}$. By multiplying with some positive $h \in C(X_A)$ with $h(x) = 1$, we can assume that $\text{supp}(f) = \text{supp}(f^\prime)$. Setting $U = \text{supp}(f)$, we have that
 		${\alpha_{\textrm{n}}}_{\restriction_{U}} = \alpha_{\textrm{n}\textrm{n}_f} = \alpha_{\textrm{m}\textrm{n}_{f^{\prime}}} = {\alpha_\textrm{m}}_{\restriction_{U}} $. Thus, $\pi_A$ is surjective and well defined. 
 		
 		Finally, it is clear that for $\textrm{n}\in N(A)$ and $U\subseteq U_\textrm{n}$ the preimage of $[\alpha_\textrm{n},U]$ by $\pi_A$ is $\{[[z\cdot \textrm{n}, U]]: z\in \mathbb{T}\}$, so $\pi_A$ is continuous.  
 	\end{proof}
 	
 	\begin{proposition}
 		Let $A$ be a unital $L^p$-operator algebra for some $p\in[1,\infty)\setminus\{2\}$, then the sequence $X_A \times \mathbb{T} \xrightarrow{i_A} \mathcal{E}_A \xrightarrow{\pi_A} \mathcal{G}_A$ is exact.
 	\end{proposition}
 	
 	\begin{proof}
 		We need to show that $\text{Im}(i_A) = \text{Ker}(\pi_A)$. First, since for any $f,g \in C(X_A)$ the partial homeomorphism $\alpha_{\textrm{n}_f}$ realized by $\textrm{n}_f$ is the identity map on the support of $f$, we have that $\text{Im}(i_A) \subseteq \text{Ker}(\pi_A)$. On the other hand, let $[[\textrm{n},x]]\in \text{Ker}(\pi_A)$, then $\alpha_\textrm{n}\restriction_{U}=\text{Id}\restriction_{U}$ for some neighborhood of $x$. Therefore, it is easy to check that $(\textrm{n},x)\approx (\textrm{n}_h,x)$ where $h\in C(X_A)$ with $h(x)=a(x)/|a(x)|=b(x)/|b(x)|$, so $(\textrm{n}_h,x)\in i_A(X_A\times \mathbb{T})$.
 	\end{proof}

 	\begin{theorem}
 		\label{main_result}
 		Let $p\in[1,\infty)\setminus\{2\}$, let $\mathcal{E}$ be a twist over a topological principal, locally compact, Hausdorff, étale groupoid $\mathcal{G}$ with compact unit space. Then, there is an isomorphism $\varphi   \colon\mathcal{E}_{F^p_{\lambda}(\mathcal{G};\mathcal{E})} \rightarrow \mathcal{E}$ such that the following diagram
 		
 		\label{main}
 		\begin{equation}
 			\xymatrix{
 				\mathcal{G}^{(0)} \times \mathbb{T} \ar@{>}[rr]^{i_{F^p_{\lambda}(\mathcal{G};\mathcal{E})}} \ar@{>}[d]^{\textrm{=}} & &
 				\mathcal{E}_{F^p_{\lambda}(\mathcal{G};\mathcal{E})} \ar@{>}[d]^{\varphi} \ar@{>}[rr]^{\pi_{F^p_{\lambda}(\mathcal{G};\mathcal{E})}}& & \mathcal{G}_{F^p_{\lambda}(\mathcal{G};\mathcal{E})} \ar@{>}[d]^{\theta} \\
 				\mathcal{G}^{(0)} \times \mathbb{T} \ar@{>}[rr]^{i} &  &\mathcal{E} \ar@{>}[rr]^{\pi}  & &\mathcal{G} }
 		\end{equation}
 		commutes, where $\theta$ is the groupoid isomorphism given in Remark \ref{exp}.
 	\end{theorem}
 	
 	\begin{proof}
 		By Theorem \ref{groupoid_rec} we have that the right vertical arrow is an isomorphism, it therefore suffices to define $\varphi$ and show that it is a groupoid homomorphism which commutes in the diagram.
 		Before defining $\varphi$ we need to make the following observations.
 		
 		Let $x\in \mathcal{G}^{(0)}$ and $\textrm{n}=(a,b)$ be an admissible pair with   $ba(x) > 0$, and let $B_\textrm{n}$ be the open bisection in $\mathcal{G}$ induced by $\textrm{n}$ given in Proposition \ref{bisection}. Then, $ba(x) = b(S(B_\textrm{n} x)^{-1})a(S(B_\textrm{n} x)) > 0$. This means that $\textrm{arg}(a(S(B_\textrm{n} x))) = -\textrm{arg}(b(S(B_\textrm{n} x)^{-1}))$ and so 
 		$$a(S(B_\textrm{n} x))/|a(S(B_\textrm{n} x))| = \overline{b(S(B_\textrm{n} x)^{-1})/|b(S( B_\textrm{n} x)^{-1})|}\,.$$

 		Now, let $(\textrm{n},x) \approx (\textrm{m},x)$, where $\textrm{n}=(a,b)$ and $\textrm{m}=(c,d)$. Then, there exist $f,g \in C(\mathcal{G}^{(0)})$ with $f(x),g(x)>0$ such that $\textrm{n}\textrm{n}_f = \textrm{m}\textrm{n}_g$, i.e., $(af,\overline{f}b) = (cg,\overline{g}d)$ which written out gives us the equalities $a(S(\gamma))f(s(\gamma)) = c(S(\gamma))g(s(\gamma))$ and  $b(S(\gamma)^{-1})\overline{f(r(\gamma))} = d(S(\gamma)^{-1})\overline{g(r(\gamma))}$ for all $\gamma \in \mathcal{G}$. Note that $B_{\textrm{n}\textrm{n}_f} = B_{\textrm{m}\textrm{n}_g}$. It follows from the above equalities that there exists a neighborhood of $B_\textrm{n}x$ where $B_\textrm{m}$ and $B_\textrm{n}$ agree. In particular $B_\textrm{m} x = B_\textrm{n} x$, and therefore we have that $a(S(B_\textrm{n} x))f(x) = c(S(B_\textrm{n} x))g(x)$ (similarly $b(S(B_\textrm{n} x)^{-1})f(x) = d(S(B_\textrm{n} x)^{-1})g(x)$). Rewriting this, we have $$a(S(B_\textrm{n} x)) = c( S(B_\textrm{n} x))\cdot g(x)/f(x)\,,$$ where $g(x)/f(x)$ is a positive constant. This implies that 
 		$$a(S(B_\textrm{n} x))/|a(S(B_\textrm{n} x))| = c(S(B_\textrm{n} x))/|c(S(B_\textrm{n} x))|\,.$$ 
 		Similarly it follows that  $b(S(B_\textrm{n} x)^{-1})/|b(S(B_\textrm{n} x)^{-1})| = d(S(B_\textrm{n} x)^{-1})/|d(S(B_\textrm{n} x)^{-1})|$.

 		Now, let us define $\varphi \colon \mathcal{E}_{F^{p}_\lambda(\mathcal{G};\mathcal{E})} \rightarrow \mathcal{E}$ by 
 		$$
 		[[ \textrm{n},x ]] \mapsto \frac{a(S(B_\textrm{n} x))}{|a( S(B_\textrm{n} x))|}\cdot  S(B_\textrm{n} x)\,,
 		$$
 		where $\textrm{n}=(a,b)$ and $x\in \mathcal{G}^{(0)}$ with $ba(x)> 0$.  From the previous observations we have that $\varphi$ is well defined.
 		
 		We then need to show that $\varphi$ is a homomorphism. Let $\textrm{n} = (a,b)$ and $\textrm{m} = (b,c)$ be two admissible pairs. Pick $x \in \mathcal{G}^{(0)}$. Then $[ \textrm{n},\alpha_{\mathrm{m}}(x) ] [ \textrm{m},x ] = [ \textrm{nm},x ]$, where $\textrm{nm} = (ac,bd)$. We first need to show that $B_\textrm{n}  B_\textrm{m} x = B_{\textrm{nm}} x$. Set $\gamma = B_\textrm{n}  B_\textrm{m}x=B_\textrm{n}\alpha_\textrm{m}(x)  B_\textrm{m}x$. We have that  $$ac(S(\gamma)) = \sum_{\eta \in \mathcal{G}_x}a(S(\gamma)S(\eta)^{-1})c(S(\eta)) = a(S(B_\textrm{n}\alpha_\textrm{m}(x) ))c(S(B_\textrm{m} x)) \neq 0\,.$$ Similarly, $bd(S(\gamma)^{-1}) \neq 0$. It follows that $B_\textrm{n} B_\textrm{m} x = B_{\textrm{nm}} x$, and using this we have that
 		$$
 		\frac{ac(S(B_{\textrm{nm}}x))}{|ac(S(B_{\textrm{nm}}x))|} = \frac{a(S(B_\textrm{n} \alpha_\textrm{m}(x)))}{|a(S(B_\textrm{n}\alpha_\textrm{m}(x)))|}\frac{c(S(B_\textrm{m} x))}{|c(S(B_\textrm{m} x))|}\,.
 		$$
 		Thus, $\varphi([[\textrm{n},\alpha_\textrm{m}(x)]]\cdot [[\textrm{m},x]])=\varphi([[\textrm{nm},x]])$. Moreover, 
 		
 		\begin{align*}
 			\varphi([[\textrm{n},x]]^{-1}) & =\varphi([[\textrm{n}^\sharp,\alpha_\textrm{n}(x)]])=\frac{b(S(B_\textrm{n} x)^{-1})}{|b( S(B_\textrm{n} x)^{-1})|}\cdot  S(B_\textrm{n} x)^{-1} =\\ 
 			& =\overline{\frac{a(S(B_\textrm{n} x))}{|a( S(B_\textrm{n} x))|}}\cdot  S(B_\textrm{n} x)^{-1}=\left( \frac{a(S(B_\textrm{n} x))}{|a( S(B_\textrm{n} x))|}\cdot  S(B_\textrm{n} x)\right) ^{-1}=\varphi([[\textrm{n},x]])^{-1}\,,
 		\end{align*}
 		thus, $\varphi$ is a groupoid homomorphism. 
 		
 		Finally, we need to show that the diagram commutes. Let $x \in \mathcal{G}^{(0)}$, $z \in \mathbb{T}$ and let $f \in C(\mathcal{G}^{(0)})$ be any function such that $f(x) = z$. Note that $[[ \textrm{n}_f,x ]] = i_{F^{p}_\lambda(\mathcal{G};\mathcal{E})}(x,z)$. Then, we have that $\varphi([[ \textrm{n}_f,x ]]) = f(x) \cdot x = i(x,z)$, as expected. Now let $[[\textrm{n},x ]] \in \mathcal{E}_{F^{p}_\lambda(\mathcal{G};\mathcal{E})}$, then we have that 
 		
 		$$
 		\pi(\varphi([[ \textrm{n},x ]])) = 
 		\pi \left(\frac{a(S(B_\textrm{n} x))}{|a(S(B_\textrm{n} x))|} \cdot S(B_\textrm{n} x)\right) = B_\textrm{n} x = \theta([[\textrm{n},x]]),
 		$$
 		where $\theta$ is given as in Remark \ref{exp}. Hence the diagram commutes.
 	\end{proof}
 	
 	% \begin{proof}
 		% Let $F^p_{\lambda}(\mathcal{G}, \alpha) \cong F^p_{\lambda}(\mathcal{H}, \beta)$, then $F^p_{\lambda}(\mathcal{G}, \mathcal{E}_\alpha) \cong F^p_{\lambda}(\mathcal{H}, \mathcal{E}_\beta)$. We need to show that $\mathcal{E}_{F^p_{\lambda}(\mathcal{H}, \mathcal{E}_\beta)} \cong \mathcal{E}_{F^p_{\lambda}(\mathcal{H}, \mathcal{E}_\beta)}$. Let $\varphi \colon F^p_\lambda(\mathcal{G},\alpha) \rightarrow F^p_\lambda(\mathcal{H},\beta)$ be an isomorphism. Then it follows form Proposition 2.13 in \cite{choi2021rigidity} that there is a $\ast$-isomorphism  $\textrm{core}(F^p_\lambda(\mathcal{G},\alpha)) \cong \textrm{core}(F^p_\lambda(\mathcal{H},\beta))$, i.e $C(\mathcal{G}^{0}) \cong C(\mathcal{H}^0)$. Then $\varphi$ sends admissible pairs to admissible pairs, and $\mathcal{H}^{0}$ is homeomorphic to $\mathcal{G}^{0}$ 
 		
 		% Then by Theorem \ref{main}, we it follows that $\mathcal{E}_\alpha \cong \mathcal{E}_\beta$ and by Remark \ref{twist_cohomology} this implies that $\alpha$ and $\beta$ are cohomologous. $\mathcal{G} \cong \mathcal{H}$ Follows from Corollary \ref{rig}.
 		
 		% \end{proof}
 	
 	Similarly to the group of cohomologous cocycles for a locally compact group, one can construct a group of isomorphic twists over an étale groupoid. We will follow  \cite[Section 5.2]{2017arXiv171010897S} in doing so.
 	
 	Let $\mathcal{G}$ be an étale groupoid and let $\mathcal{E}$ and $\mathcal{F}$ be two twists over $\mathcal{G}$. Two twists are said to be \emph{properly isomorphic} if there exists a groupoid isomorphism $\phi$ such that the following diagram
 	
 	\begin{equation}
 		\xymatrix{
 			\mathcal{G}^{(0)} \times \mathbb{T} \ar@{>}[r] \ar@{>}[d]^{=} & \mathcal{E} \ar@{>}[d]^{\phi} \ar@{>}[r] & 
 			\mathcal{G} \ar@{>}[d]^{=}
 			\\
 			\mathcal{G}^{(0)} \times \mathbb{T} \ar@{>}[r]^{i} &
 			\mathcal{F} \ar@{>}[r]  &  \mathcal{G}}
 	\end{equation}
 	commutes.
 	We write $[\mathcal{E}]$ for the equivalence class of the twist $\mathcal{E}$ over $\mathcal{G}$ containing all twists over $\mathcal{G}$ that are properly isomorphic to $\mathcal{E}$. The collection of equivalence classes of proper isomorphic twists on $\mathcal{G}$ is denoted by $\textrm{Tw}(\mathcal{G})$.
 	
 	Let $\mathcal{E}$ and $\mathcal{E}^\prime$ be two twists over $\mathcal{G}$. On the set 
 	$$
 	\mathcal{E} \times_\pi^{\pi^{\prime}} \mathcal{E}^\prime = \{(\varepsilon,\varepsilon^\prime) \in \mathcal{E} \times \mathcal{E}^\prime \colon \pi(\varepsilon) = \pi^\prime(\varepsilon^\prime)  \}, 
 	$$
 	define the following equivalence relation: $(\varepsilon,\varepsilon^\prime) \sim (\delta,\delta^\prime)$ if and only if there exists $z\in \mathbb{T}$ such that $z \cdot \varepsilon = \delta$ and $\overline{z} \cdot \varepsilon^\prime = \delta^\prime$. The quotient $\mathcal{E} \ast \mathcal{E}^\prime = \mathcal{E} \times_\pi^{\pi^{\prime}} \mathcal{E}^\prime / \sim$ is in fact a twist over $\mathcal{G}$ given by
 	$$
 	\mathcal{G}^{(0)} \times \mathbb{T} \xrightarrow{i \ast i^\prime} \mathcal{E} \ast \mathcal{E}^\prime \xrightarrow{\pi \ast \pi^\prime} \mathcal{G}
 	$$
 	where $(i \ast i)(x,z) = [i(x,z),i^\prime(1,x)]$ and $(\pi\ast\pi^\prime)([\varepsilon,\varepsilon^\prime]) = \pi(\varepsilon)$. The collection $\textrm{Tw}(\mathcal{G})$ forms an abelian group under the group operation given by $\mathcal{E} + \mathcal{E}^\prime = [\mathcal{E} \ast \mathcal{E}]$ and identity given by the class of the trivial twist. By \cite[Proposition 2.13]{choi2021rigidity}, an isometric isomorphism of two $L^p$-operator algebra induces an injective $C^\ast$-homomorphism of the $C^\ast$-cores of the algebras. We therefore have the following corollary from Theorem \ref{main_result}.
 	
 	\begin{corollary} Let $p \in [1, \infty) \setminus \{2\}$.
 		Let $\mathcal{G}$ and $\mathcal{H}$ be topological principal, locally compact, Hausdorff, étale groupoids with compact unit spaces, let $\mathcal{E}$ be a   twist over $\mathcal{G}$ and $\mathcal{F}$ be a twist over $\mathcal{H}$. Then there is an isometric isomorphism $F^p_{\lambda}(\mathcal{E};\mathcal{G}) \cong F^p_{\lambda}(\mathcal{E};\mathcal{H})$ if and only if there is an isomorphism of groupoids $\mathcal{G} \cong \mathcal{H}$, and $[\mathcal{E}] = [\mathcal{F}]$ in $\textrm{Tw}(\mathcal{G})$.
 	\end{corollary}
 	
 	There is an group isomorphism between $H^2(\mathcal{G},\mathbb{T})$ \cite[Definition 1.16]{RenaultJean1980AGAt} and the subset of $\textrm{Tw}(\mathcal{G})$ consisting of twists by continuous sections \cite[Section 4]{kumjian_1986}. We therefore have the  following corollary, analogous to Theorem \ref{main-group}, the main result in Section 4. 
 	
 	\begin{corollary}
 		\label{twist_rig} Let $p \in [1, \infty) \setminus \{2\}$.
 		Let $\mathcal{G}$ and $\mathcal{H}$ be topologically principal, locally compact, Hausdorff, étale groupoids with compact unit spaces, let $\sigma$ be a 2-cocycle for $\mathcal{G}$ and $ \rho$ be a 2-cocycle for $\mathcal{H}$. Then there is an isometric isomorphism $F^p_{\lambda}(\mathcal{G}, \sigma) \cong F^p_{\lambda}(\mathcal{H}, \rho)$ if and only if there is an isomorphism of groupoids $\mathcal{G} \cong \mathcal{H}$ and $\sigma$ and $\rho$ are cohomologous.
 	\end{corollary}
 	
 	Observe that the above result is far from being a generalization of Theorem \ref{main-group} as locally compact groups are neither étale nor topologically principal, and even though discrete groups are étale, they are not topological principal. In fact, the only group that can be topological principle is the trivial group.

 	\section*{Acknowledgments}
 	We would like to thank Jan Gundelach for the careful reading of the first draft of the paper and his comments. Finally, we want to thank the anonymous referee for his or her thorough feedback.

\end{document}